\documentclass[11pt]{amsart}
\usepackage[utf8]{inputenc}
\usepackage[english]{babel}
\usepackage{cite}
\usepackage{amsthm,amsmath,amssymb,hyperref,amsfonts,tikz,adjustbox}
\usepackage{mathrsfs}
\usetikzlibrary{cd}
\numberwithin{equation}{section}
\usepackage[shortlabels]{enumitem}
\usepackage{imakeidx}

\usepackage{tikz}
\usetikzlibrary{cd}
\usetikzlibrary{calc} 
\usetikzlibrary{decorations.pathmorphing} 

\tikzset{curve/.style={settings={#1},to path={(\tikztostart)
    .. controls ($(\tikztostart)!\pv{pos}!(\tikztotarget)!\pv{height}!270:(\tikztotarget)$)
    and ($(\tikztostart)!1-\pv{pos}!(\tikztotarget)!\pv{height}!270:(\tikztotarget)$)
    .. (\tikztotarget)\tikztonodes}},
    settings/.code={\tikzset{quiver/.cd,#1}
        \def\pv##1{\pgfkeysvalueof{/tikz/quiver/##1}}},
    quiver/.cd,pos/.initial=0.35,height/.initial=0}

\tikzset{tail reversed/.code={\pgfsetarrowsstart{tikzcd to}}}
\tikzset{2tail/.code={\pgfsetarrowsstart{Implies[reversed]}}}
\tikzset{2tail reversed/.code={\pgfsetarrowsstart{Implies}}}
\tikzset{no body/.style={/tikz/dash pattern=on 0 off 1mm}}

\newtheorem{theorem}{Theorem}[section]
\newtheorem{corollary}[theorem]{Corollary}
\newtheorem{lemma}[theorem]{Lemma}
\newtheorem{proposition}[theorem]{Proposition}

\theoremstyle{definition}
\newtheorem{definition}[theorem]{Definition}
\newtheorem{example}[theorem]{Example}

\theoremstyle{remark}
\newtheorem{remark}[theorem]{Remark}

\newcommand{\m}{{}^{-1}}  
\newcommand{\s}{\sigma}
 \newcommand{\0}{\theta} \newcommand{\de}{\delta}

\newcommand{\kpar}{{\kappa }_{\rm par}}
\newcommand{\niton}{\not\owns}

\title[Twisted partial group algebra and topological dynamics]{Twisted partial group algebra and  related topological partial dynamical system}



\author{Mikhailo Dokuchaev}
\address{Departamento de Matem{\'a}tica, Universidade de S\~ao Paulo,
Rua do Mat\~ao, 1010, 05508-090 S\~ao Paulo, Brazil}
\email{dokucha@ime.usp.br}  

\author{Emmanuel Jerez}
\address{Departamento de Matem{\'a}tica, Universidade de S\~ao Paulo,
Rua do Mat\~ao, 1010, 05508-090 S\~ao Paulo, Brazil}
\email{ejerez@ime.usp.br}

\begin{document}


\begin{abstract} 
Given  a group $G,$ a field $\kappa $ and a factor set  $\sigma $  of some partial projective $\kappa$-representation of $G,$   we consider a topological partial dynamical system $(\Omega_\sigma, G, \hat{\theta}),$ in which  $\Omega_\sigma $  is a compact totally disconnected Hausdorff space and $\sigma $ is a twist for $\hat{\theta},$ such that $\kpar ^{\sigma} G$
can be seen as a crossed product of the form     
${\mathscr L}(\Omega_\sigma )  \rtimes_{({\hat{\theta}}, \sigma)} G,$ where ${\mathscr L}(\Omega_\sigma )$ stands for the $\kappa $-algebra of the locally constant functions $\Omega_\sigma \to \kappa .$ 
The space $\Omega_\sigma $ is homeomorphic  with the spectrum of a unital commutative subalgebra  of $\kpar ^{\sigma} G$ generated by idempotents. We describe $\Omega_\sigma $ as a subspace of the  Bernoulli space $2^G$ and study when the spectral partial action   $\hat{\theta}$ is topologically free, which is a property related to important information on the ideals of   $\kpar ^{\sigma} G.$   We also show how to generate arbitrary idempotent factor sets of  $G$ and give a condition on them which guarantees the topological freeness of $\hat{\theta}$. 
Furthermore, inspired by the semigroup $\mathcal{S}(G)$ defined by R. Exel to control partial actions and partial representations of $G$ 
and its relation to the partial group algebra  $\kpar G,$ we characterize the twisted partial group algebra $\kpar ^{\sigma} G$ as an algebra generated by a $\kappa$-cancellative inverse semigroup, the latter being defined using elements of $\Omega_\sigma .$   If $\Omega_\sigma $ is discrete, we prove that $\kpar ^{\sigma} G$ is a product of matrix algebras over twisted subgroup algebras, extending a known  result  for $\kpar G$ with finite $G.$ For the infinite dihedral group \(D_{\infty}\) we describe how to obtain all partial factor sets of \(D_{\infty}\), up to equivalence.
\end{abstract}

\maketitle

\section{Introduction}

Group algebras rule group representations and, more generally, given a $2$-cocycle $\sigma : G \times G \to {\kappa }^*$ of a group $G$ with values in the multiplicative group ${\kappa }^*$ of a field $\kappa ,$ the projective $\kappa $-representations of $G$ are governed by the twisted group algebra $\kappa \ast _{\s} G.$ As to partial representations of a group $G,$ introduced in the theory of operator algebras \cite{QR}, \cite{exe} (see also \cite{Mc} and\cite{DEP})  we know that the partial group algebra $\kpar G$ of $G$, defined by R. Exel initially in the context of $C^*$-algebras \cite{exe}, is responsible for their control. The study of partial projective group representations began in \cite{DN} with further developments  in  
\cite{DoNoPi}, \cite{NovP}, \cite{HGGLimaHPinedo}, \cite{Pi5}, \cite{DoSa}. Recently, the concept of the twisted partial group algebra     $\kpar ^{\sigma} G$ was offered in \cite{Dokuchaev-Jerez:twisted:2024}, where $\sigma $ is a factor set of some partial projective representation of $G$. The representations of  $\kpar ^{\sigma} G$ in a $\kappa $-algebra $R$ are in one-to-one correspondence with the partial projective representations of $G$ in $R$, whose factor set is $\sigma $ (partial projective $\sigma $-representations of $G$). One of the results of \cite{Dokuchaev-Jerez:twisted:2024} says that the algebra $\kpar ^{\sigma} G$
is isomorphic to the crossed product  $ B^\sigma \rtimes_{({\theta}^{\sigma}, \sigma)} G$ (also denoted simply by   
$ B^\sigma \rtimes_{\sigma} G$), where $ B^\sigma $ is a commutative subalgebra of  $\kpar ^{\sigma} G$ generated by idempotents and ${\theta}^{\sigma}$ is a unital partial action of $G$ on  $ B^\sigma ,$ such that $\sigma $ is  a twist of   ${\theta}^{\sigma}.$ This extends an earlier result proved for $\kpar G$ in \cite{DE1}.

Our purpose is to give a better understanding of $\kpar ^{\sigma} G,$  study the  to\-po\-lo\-gi\-cal partial dynamical system related to  ${\theta}^{\sigma}$ and contribute to the know\-led\-ge about the factors sets of partial projective group representations.

In Section~\ref{sec:SomeRemarks} we provide some  preliminary information  on spectra of commutative unital algebras generated by idempotents and, more generally, on compact totally disconnected Hausdorff spaces. In particular, we shall use the Gelfand-like Theorem~\ref{t: JacobsonKeimel}, which is a known fact, stating that  a unital commutative $\kappa $-algebra  $A$ generated by idempotents is isomorphic to the algebra of the locally constant functions 
${\rm Spec}\, (A) \to \kappa .$

 In Section~\ref{sec:Basics} we recall some background on partial projective group representations, twisted partial group actions and the twisted partial group algebra.

Since $B^{\sigma }$ is a commutative unital $\kappa $-algebra generated idempotents, by the above-mentioned Gelfand-like Theorem  it can be identified  with the algebra of the locally constant $\kappa $-valued functions on ${\rm Spec}\, (B^{\sigma}) ,$ and the partial action ${\theta}^{\sigma}$ of $G$ on 
$B^{\sigma }$ corresponds to a topological partial action  $\hat{\theta}$ of $G$ on  ${\rm Spec}\, (B^{\sigma}) $ (the spectral partial action). In 
Section~\ref{sec:Spectral} we identify ${\rm Spec}\, (B^{\sigma}) $ with a subspace $\Omega_\sigma$ of the  Bernoulli space $2^G$  and describe the spectral partial action $\hat{\theta}$
as a restriction of  the well known Bernoulli partial action of $G$ (see \cite[Definition 5.12]{E6}). We consider $2^G$ as the power set of $G$ and  define the subset $ \mathcal{P}_\sigma \subseteq  2^G,$ called the set of $\sigma$-prohibitions, and characterize the elements of   $\Omega_\sigma$ as those subsets of $G$,
which contain $1_G$ and do not  contain   $\sigma$-prohibitions. We also prove some preliminary facts on $B^{\sigma}$ and  $\Omega_\sigma$. In particular, we describe $ B^\sigma$ in terms of generators and relations in Proposition~\ref{l: B Isigma quotient}. Besides the commutativity condition, there are relations associated to equalities  of the form  $\sigma(g,h)=0$, $(g,h\in G),$ and those which occur when $\sigma $ does not satisfy the $2$-cocycle equality for some ordered triple $(g,h,t)$ of elements of $G.$ Similarly, a part of the  $\sigma$-prohibitions for the elements of $\Omega_\sigma$  is related to the equalities  $\sigma(g,h)=0$ and called of type I, and another part,  the $\sigma$-prohibitions of type II, is associated to the failure  of  the $2$-cocycle identity, i.e.  $\partial(\sigma)(g,h,t) \neq 1,$ where $\partial $ stands for the coboundary operator. Then  $$\mathcal{P}_\sigma = 
\mathcal{P}_\sigma^{\mathrm{I}} \cup \mathcal{P}_\sigma^{\mathrm{II}},$$ where $\mathcal{P}_\sigma^{\mathrm{I}}$ denotes the 
$\sigma$-prohibitions of type I, whereas $\mathcal{P}_\sigma^{\mathrm{II}}$ stands for those  of type II.

The partial group algebra $\kpar G$ can be seen as the semigroup algebra $\kappa  \mathcal{S}(G),$ where $\mathcal{S}(G)$ is the inverse semigroup defined by R. Exel  in  \cite{exe} in terms of generators and relations and, besides being important for partial representations of $G$,
the monoid $\mathcal{S}(G)$ also controls the partial actions of $G$ on sets \cite[Theorem 4.2]{exe}. It is known \cite{KL} that  $\mathcal{S}(G)$ is isomorphic to the Szendrei expansion of $G$ (see \cite{Sze}), the latter being also isomorphic to the Birget-Rhodes prefix expansion of $G$  (see \cite{BR1}). The description of 
$\kpar G$ as a semigroup algebra is extended  to  $\kpar ^{\sigma} G$ in Section~\ref{sec:CancellativeMon}. For this purpose we introduce  a $\kappa$-monoid  $\mathcal{S}^{\sigma}(G)$ determined by the $\sigma$-prohibitions. In
Proposition~\ref{prop:Cancellative} we show that 
    $\mathcal{S}^{\sigma}(G)$ is  $\kappa$-cancellative.
 Then, using the multiplication rule of  $\mathcal{S}^{\sigma}(G) ,$ we define a $\kappa $-algebra $\mathcal{A}_{\sigma}$ such that there is a  monomorphism  of
  $\mathcal{S}^{\sigma}(G) $ into the multiplicative semigroup of $\mathcal{A}_{\sigma}.$ This is followed by 
    Theorem~\ref{t: kparsG isomorphism}, which exhibits an algebra isomorphism  $\kappa_{par}^{\sigma}G \to \mathcal{A}_{\sigma}$.
    
 It was proved  in \cite{DEP} that if $G$ is a  finite group then for the partial group algebra the following decomposition holds
 $$\kappa_{par} G \cong \bigoplus_{i} M_{n_i}(\kappa H_i),$$ where  $i$ runs over the indexes of the connected components  of a groupoid $\Gamma (G)$ associated to $G$ and $H_i$ is  the isotropy group of
 an object  of the $i$th connected component of  $\Gamma (G).$ We extend this fact to $\kappa_{par}^{\sigma}G$ in Section~\ref{sec:Another}. 
 With this goal we assume that  
 $\Omega_\sigma$ is discrete. If $G$ is finite, then  $\Omega_\sigma$ is discrete for any $\sigma $, so that we include the case of an arbitrary finite group $G.$ Then we define the $\kappa $-algebra $\mathcal{R}^{\sigma}(G),$ which is a twisted analogue of the groupoid algebra 
 $\kappa \Gamma(G),$ and prove in  Proposition~\ref{prop:kpariso2} that the algebras 
     $ \kappa_{par}^{\sigma}G $ and $\mathcal{R}^{\sigma}(G)$ are isomorphic. Next we introduce a subset
     $\Gamma^{\sigma}(G) $ of $ \mathcal{R}^{\sigma}(G)$, which is a groupoid with the product inherited from $\mathcal{R}^{\sigma}(G)$.  For each connected component $ \Gamma_i^{\sigma}(G)$ of  $\Gamma^{\sigma}(G),$  we define the algebra   $\mathcal{R}^{\sigma}_{i}(G)$ as $\kappa$-subspace of $\mathcal{R}^{\sigma}(G)$ generated by $\Gamma_i^{\sigma}(G)$, where $i$ runs over the index set $I$ of the connected components of  $\Gamma^{\sigma}(G)$. Then 
$\mathcal{R}^{\sigma}(G) = \bigoplus_{i\in I} \mathcal{R}_{i}^{\sigma}(G),$ and we use this decomposition to establish Theorem~\ref{t: groupoid algebra isomorphism}, which says that if $\Omega_\sigma$ is discrete, then  
$$\kappa_{par}^{\sigma}G \cong \bigoplus_{i\in I} M_{n_i}(\kappa^{\sigma_i} H_i),$$
where $H_i$ is the isotropy group of an object of $ \Gamma_i^{\sigma}(G),$  $\sigma _i$ is a global $2$-cocycle of $H_i,$ obtained as  the restriction of $\sigma $ to $H_i \times H_i,$ and $\kappa^{\sigma_i} H_i$ is the usual twisted group algebra of $H_i.$
		
In Section~\ref{sec:S4-Invariance} we temporarily assume that the field $\kappa $ is algebraically closed and construct an action of the symmetric group $S_4$ on $G^3,$ such that any $\gamma \in S_4$ either preserves  $\partial(\sigma)(x,y,z)$ or inverts it (see Proposition~\ref{prop:S4-action on delta}). The latter option occurs exactly when   $\partial(\sigma)(x,y,z)\neq 0$ and $\gamma$ is an odd permutation. It follows that  $\partial(\sigma)(x,y,z) = 1$ if and only if $\partial(\sigma)(\gamma \triangleright (x, y, z)) = 1$ for all $\gamma \in S_4,$ which implies that the set $\mathcal{P}_\sigma^{\mathrm{II}}$  is invariant under the action of $S_4.$ As a byproduct we obtain in Proposition~\ref{prop:refinement} a refinement with respect to the defining relations of $B^{\sigma},$ provided that $\sigma$ is totally defined (and, as assumed above, $\kappa$ is algebraically closed).
		
Section~\ref{sec:TopolFree} is dedicated to the study of the topological freeness of the spectral partial action (see Definition~\ref{def:TopolFree}). Corollary~\ref{cor:TopolFreeImportance2} states that if the spectral partial action $\hat{\theta}$ is topologically free, then $\mathcal{I} \cap  B^{\sigma}  \neq 0$ for each non-zero ideal $\mathcal{I}$ of  $\kpar ^{\sigma} G$.
Thus, the to\-po\-lo\-gi\-cal freeness of $\hat{\theta}$ guarantees a relevant property of the ideals of the twisted partial group algebra \(\kappa_{par}^{\sigma}G\).
It is rather easy to see that  the spectral partial action of $G$ is topologically free, if  the group $G$ is  torsion-free (see Proposition~\ref{prop:torsion-free}).
On the other hand, the to\-po\-lo\-gy of $\Omega_{\sigma}$ is discrete when $G$ is finite (see   Remark~\ref{r: Omega sigma is discrete when G is finite}), so that the main problem is to study the topological freeness of $\hat{\theta}$ when $G$ is an infinite group with non-trivial torsion part.
After proving several facts on the topological space  $\Omega_\sigma$  and its elements, related to the topological freeness of $\hat{\theta},$ we obtain the main result of the section, Theorem~\ref{t: not topologically free charaterization}, which gives two conditions, each of them being equivalent to the fact that $\hat{\theta}$ is not topologically free.
One of the conditions is the  existence of an isolated fixed point and the other one is the existence of a finite fixed point with an additional finiteness condition.
	
	It is known that any factor set $\sigma $ of a partial projective  representation of a group $G$ (partial factor set of $G$) is equivalent to a product of an idempotent partial factor set of $G$ and a totally defined partial factor set of $G$, i.e. a partial factor set  whose domain is $G\times G$ (see \cite{DoNoPi}). Therefore, the problem of finding   all partial factor sets of a group $G$ up to equivalence is reduced to the determination of the idempotent partial factor sets and the totally defined 
	ones. In Section~\ref{sec:Idempotent} we show how to construct an arbitrary idempotent partial factor set of a group $G$ from  a subset of $ G \times G,$ which satisfy a condition of symmetry \eqref{eq: General idempotent condition}. This is achieved by introducing the so-called diagonal (partial) factor sets 
	in Section~\ref{sec:DiagonalIdempotent} and the lateral (partial) factor sets in Section~\ref{sec:LateralIdempotent}. Then Theorem~\ref{p: unique decomposition of idempotent factor sets} states, in particular, that any idempotent (partial) factor set of $G$ can be uniquely decomposed into a product of a diagonal  factor set and a lateral one. As to the topological freeness of $\hat{\theta}$, we give a sufficient condition for it in the case of a diagonal factor set in Proposition~\ref{prop:DiagonalTopolFree}
	and for the case of a lateral factor set in 
	Proposition~\ref{prop:LateralTopolFree}. This allows us to give a sufficient condition for the 
	topological freeness of $\hat{\theta}$ in the case of a general idempotent factor set in Corollary~\ref{cor:IdempotentTopolFree}. In particular, we conclude in Corollary~\ref{cor:IdempotentNullTopolFree} that if $\sigma$ is an idempotent factor set of $G$ such that 
   $|\operatorname{Null}(\sigma)| < |G|,$  then
   $\hat{\theta}$ is topologically free, where
   $ \operatorname{Null}(\sigma):= \{ (x,y) \in G \times G : \sigma(x,y) = 0 \}.$
    It is interesting to observe that if  $\sigma$ is an idempotent factor set, then the elements of $\Omega_{\sigma}$ are determined only by $\sigma$-prohibitions of type I (see Lemma~\ref{sec:OnlyTypeI}).
    
 The obtained results on the idempotent factor sets are applied in  Section~\ref{sec:InfDihedral} to  the  infinite dihedral group $D_\infty.$ Dealing with  $\sigma$-prohibitions of type I we conclude that a diagonal factor set of $D_\infty$ is determined by an arbitrary pair of functions  $ \mathbb{Z}^{+} \to \{ 0,1 \}$ and $ \mathbb{Z} \to \{ 0,1 \}$, so that the subsemigroup $pm(D_\infty)_{\operatorname{diag}} \subseteq pm(D_{\infty})$  of all diagonal factor sets of $D_\infty$ is isomorphic to the semigroup 
 $\{0,1 \}^{\mathbb{Z}^{+} \sqcup \, \mathbb{Z}}$
 (see Lemma~\ref{l: diagonal factor sets of DInf}).
 Analogously the semigroup of all lateral factor sets of $D_\infty$ is isomorphic  to
 $\{0,1 \}^{ \mathbb{Z}^{+} \times \mathbb{Z}^{+} \sqcup \mathbb{Z}^{+} \times \mathbb{Z} }$ (see 
	Proposition~\ref{p: lateral factor sets of DInf characterization}). 
	 Then we deduce in Proposition~\ref{p: idempotent factor sets of DInf} that an arbitrary pair 
     of functions \((\nu, \omega)\), where \(\nu: \mathbb{Z}^{+} \sqcup \mathbb{Z} \to \{0,1 \}\) and \(\omega: \mathbb{Z}^{+} \times \mathbb{Z}^{+} \sqcup \mathbb{Z}^{+} \times \mathbb{Z} \to \{0,1 \}\), 
     determine an idempotent factor set of \(D_{\infty}\) and  all idempotent factor sets of \(D_{\infty}\) can be obtained this way. Finally, from  \cite[Theorem 5.4]{HGGLimaHPinedo}  and Proposition~\ref{p: idempotent factor sets of DInf} we conclude in Theorem~\ref{teo:InfDihedtral} that, up to equivalence, any factor set of \(D_{\infty}\) can be obtained from a pair of functions \((\nu, \omega)\), where \(\nu: \mathbb{Z}^{+} \sqcup \mathbb{Z} \to \{0,1 \}\) and \(\omega: \mathbb{Z}^{+} \times \mathbb{Z}^{+} \sqcup \mathbb{Z}^{+} \times \mathbb{Z} \to \kappa \), and taking all such pairs we get all factor set of  \(D_{\infty}\), up to    equivalence, and the given proofs explain how to do that.

For a detailed initial background on partial group actions and partial group representations the reader is referred to R. Exel's book \cite{E6}, where some of the remarkable applications to $C^*$-algebras can be also found. Besides a series of  early $C^*$-algebraic successful employments,
partial actions  were used  
to tiling semigroups \cite{KL1},  to invesre semigroups  \cite{ChLi}, \cite{St1}, to profinite topology of groups \cite{Incollection_Coulbois_2002}, to Thompson's groups \cite{Birget},  to graded algebras \cite{DE1}, \cite{DES1}, to Hecke algebras \cite{E3},  to restriction semigroups \cite{CornGould}, \cite{Kud}, to Leavitt path algebras \cite{GR}, 
\cite{Goncalves-Oinert-Royer:2014:JA},  to algebras associated with separated graphs and paradoxical decompositions \cite{AraE1}. More recent applications include those   to convex subshifts and  related algebras \cite{AraL},  
to topological higher-rank graphs 
\cite{RenWil}, to Matsumoto and Carlsen–Matsumoto C*-algebras of arbitrary subshifts  \cite{DE2}, to Steinberg algebras \cite{BeuGon2}, \cite{HazratLi}, to ultragraph C*-algebras \cite{GR3}, to
embeddings of inverse semigroups \cite{Khry1}, to Ehresmann semigroups \cite{KudLaan} and to expansions of monoids in the class of two-sided restriction monoids \cite{Kud2}. 
More information around partial actions and their applications can be consulted in the survey articles \cite{Ba} and \cite{D3} and the  references therein.

In this work $\kappa$ will denote an arbitrary field,  except Section~\ref{sec:S4-Invariance}, in which  $\kappa$ will be assumed to be an algebraically closed field.
Algebras will be associative and unital, i.e. possessing a multiplicative identity element.
Moreover, every homomorphism of  algebras will respect unity elements. Furthermore, by a factor set we shall mean a factor set of a partial projective representation of a group $G$ over $\kappa$.

\section{Some remarks on commutative algebras generated by idempotents}\label{sec:SomeRemarks}
    
 We notice in this section some  basic facts about an arbitrary commutative unital $\kappa$-algebra $A$ generated by idempotents and locally constant functions. 

\begin{remark}\label{r: prime=maximal} 
If ${\mathfrak p}$ is a prime ideal of $A,$ then $A/\mathfrak{p} \cong \kappa.$ In particular, every prime ideal of $A$ is maximal. Indeed, any factor algebra of $A$ is a commutative $\kappa$-algebra generated by idempotents. Since  $A/\mathfrak{p}$ is an integral domain, it contains no non-trivial idempotents and thus   $A/\mathfrak{p}$ must be isomorphic to $\kappa.$
\end{remark}

By \cite[Theorem 1.16]{goodearl1979neumann} the fact that every prime ideal of $A$ is maximal, also follows from the next:

\begin{proposition}[Proposition 5.14\cite{Dokuchaev-Jerez:twisted:2024}] 
    The ring $A$  is von Neumann regular.
\end{proposition}

Since $A$ is a commutative von Neumann regular algebra then by \cite[Proposition 3.12]{goodearl1979neumann} follows that ${\rm Spec }\, A$ is a compact  totally disconnected Hausdorff space. Here ${\rm Spec }\, A$ stands for the  prime spectrum of $A$ endowed with the Zariski topology, which in our case coincides with the maximal spectrum of $A.$ Recall that the closed subsets in the Zariski topology are of the form 
\[
   V(I)=\{\mathfrak{p} \in {\rm Spec }\, A \, : \, \mathfrak{p} \supseteq I\},
\]
where $I$ is an ideal  of $A$. If $I= aA,$ $(a\in A)$, then we shall simply write $V(I)=V(a).$ Notice also that it is well-known that the  sets 
\[
   D(a) = \{ {\mathfrak p} \in {\rm Spec }\, A \, : \, {\mathfrak p}\not\ni a \}, \;\;\; (a\in A),
\]
form a basis of the Zariski topology of ${\rm Spec }\, A.$ 

Let $\hat{A}$ be the set of all non-zero $\kappa$-algebra homomorphisms $\varphi:A\to \kappa.$ In view of Remark~\ref{r: prime=maximal}, there is a bijection $\phi $ from   $\hat{A}$  onto  ${\rm Spec }\, A,$ which sends  $\varphi$ to its kernel $\mathfrak{m}_{\varphi}$. We may consider  $\hat{A}$ as a subset of $\kappa^A,$ take the discrete topology on $\kappa$ and the product topology of  $\kappa^A.$ Then, $\hat{A}$ becomes a topological  space with the induced topology. It is easy to see that the sets of the form 
\[
   Z(a,\alpha)= \{\varphi\in \hat{A} \, :\, \varphi(a)=\alpha \},
\]
where $a\in A$ and $\alpha \in \kappa,$ form a subbase of the topology of $\hat{A}.$ In addition, each $Z(a,\alpha)$ is readily seen to be closed in  $\hat{A}.$
\begin{lemma}\label{r: homeo} 
    The map $\phi : \hat{A} \to {\rm Spec }\, A$, $\varphi \mapsto \mathfrak{m}_{\varphi}= {\rm Ker}(\varphi),$ is a homeomorphism.
\end{lemma}
\begin{proof} 
    For every $a\in A$ the set  ${\phi}\m (D(a))$ coincides with the complement $Z(a,0)^c$ of $Z(a,0),$ which  is open  in  $ \hat{A} ,$  showing that $\phi $ is continuous.  In order to see that ${\phi}\m$ is also continuous, notice that $\phi (Z(a,\alpha)) = V(a-\alpha)$ for any $a\in A$ and $\alpha \in \kappa.$ Due to the fact that  $A$ is von Neumann regular, $(a-\alpha )A= eA$ for some idempotent $e\in A.$ Therefore, $V(a - \alpha) = V(e) = D(e-1)$ which is open in  ${\rm Spec }\, A.$
\end{proof}

The compact open subsets of   ${\rm Spec }\, A$ can be easily described: 

\begin{lemma}\label{r: klopen} 
The collection $\{ D(e) \},$ where $e$ runs over the idempotents of $A,$ consists of all compact  open subsets of ${\rm Spec }\, A$ and forms a base of the topology of  ${\rm Spec }\, A.$ 
\end{lemma}
\begin{proof}
Indeed, using again the fact that $A$ is von Neumann regular, we see that for any $a\in A$ there exists an idempotent $e_a\in A$ such that  $aA=e_aA.$ Hence, $D(a)= D(e_a).$ Moreover, for any idempotent $e\in A,$ we have that   $ D(e)  = V(e-1),$ showing that $ D(e) $ is compact and open. Furthermore, our collection  $\{ D(e) \}$ forms a base for the topology of     ${\rm Spec }\, A,$ as so does the collection of all $\{ D(a) \},$ $a\in A.$

Let now  $F$  be an arbitrary compact open subset in     ${\rm Spec }\, A.$ Then, $F= \bigcup_{i=1}^k D(e_i)$ for some idempotents $e_1, \ldots, e_k\in A.$ Consequently,    
\[
   F= \bigcup_{i=1}^k V(e_i-1) = V(\prod_{i=1}^k (e_i-1))= V(e) = D(e-1),
\]
where $e\in A$ is an idempotent such that $\prod_{i=1}^k (e_i-1)A = eA,$ using once more that $A$ is Von Neumann regular.
\end{proof}

Being  interested in continuous functions on ${\rm Spec} \, A$,  we recall some general facts about functions on an arbitrary compact  totally disconnected  Hausdorff topological space $X$. Notice that a compact Hausdorff space is
totally disconnected if and only if  it possesses a base of open sets which are  also closed (and hence compact) \cite[Theorem 29.7]{willard}. A function $f:X\to \kappa$ is called locally constant if every point of $X$ has a neighborhood on which $f$ is constant.
Take the discrete topology on $\kappa.$ Then it is readily seen that a function $f:X\to \kappa$ is continuous if and only if it is locally constant. An example of a locally constant function is the characteristic function $1_E$ of a compact  open  set $E.$ 
If  $f:X\to \kappa$ is locally constant, then  $f\m(\alpha )$ is easily seen to be compact and open for every $\alpha  \in \kappa$. The set
\[
   {\rm supp}(f)= \{ x\in X \, : \, f(x)\neq 0\}
\]
is called the support of $f.$ Observe that  the support of any locally constant  function $f: X\to \kappa$ is compact and open. The open cover  $\bigcup _{\alpha \neq 0}f\m(\alpha )$  of ${\rm supp}(f)$ has a finite open subcover thanks to the compactness   of ${\rm supp}(f).$ Consequently,  there exists $\alpha _1, \ldots, \alpha _k \in \kappa$ such that $$f = \sum_{i=1}^k \alpha _i1_{E_i},$$ where  $E_i = f\m(\alpha _i).$ Conversely, for any $\alpha _1, \ldots , \alpha _k\in \kappa$ and any compact open sets $E_1, \ldots, E_k,$ the function $f = \sum_{i=1}^k \alpha _i 1_{E_i}$ is   continuous. The algebra  ${\mathscr L}(X)$  of all locally constant functions $X\to \kappa$ with respect to the point-wise multiplication and point-wise addition  is clearly  a commutative unital algebra generated by idempotents.

We shall need the next   known interesting fact  (see Remark~\ref{r: JacobsonKeimel} below), for which we give  a short  proof for  the sake of completeness.
\begin{theorem}\label{t: JacobsonKeimel} 
    Let $A$ be a commutative unital $\kappa$-algebra  generated by idempotents. Then the map $\kappa _A: A\ni a \mapsto \hat{a}\in {\mathscr L}(\hat{A}),$ $\hat{a}(\varphi)=\varphi (a), \varphi\in \hat{A},$ is an isomorphism of $\kappa$-algebras.
\end{theorem}
\begin{proof}
    It is easy to see that $\hat{a}$ is continuous and $a\mapsto \hat{a}$ is a homomorphism of $\kappa$-algebras. It is injective, because if $\hat{a}=\hat{b},$ with $a,b\in A,$ then using $\phi$ we see that $a-b$ lies in any maximal ideal of $A.$ Consequently, $a-b$ is contained in the Jacobson radical of $A$ which is  $\{0\} $ by \cite[Corollary 1.2]{goodearl1979neumann}. For the surjectivity note that ${\rm supp}\, \hat{e} = \phi\m(D(e))$  and $\hat{e} =1_{\phi\m(D(e))}$ for every  idempotent  $e\in A.$ By Lemmas~\ref{r: homeo}  and~\ref{r: klopen}, any compact open subset of $\hat{A}$ is of the form $\phi\m(D(e)),$ and we are done, as the characteristic functions of the compact open sets linearly span $\hat{A}.$   
\end{proof}

 \begin{remark}\label{r: JacobsonKeimel} 
     Theorem~\ref{t: JacobsonKeimel} can be obtained from  a much more general result stated for non-necessarily commutative  bigerular rings \cite[Theorem IX.6.1]{Jacobson:StructureOfRings:1956}, observing that the prime spectrum of a unital commutative von Neumann regular ring coincides with its primitive spectrum. Theorem~\ref{t: JacobsonKeimel}  can  also be obtained from Keimel's result \cite[Theorem I]{Keimel}, stated for algebras over a commutative ring $R$ without $R$-torsion (i.e. $ra=0 \Rightarrow r=0$ or $ a=0, r\in R, a\in A$),  and using the space of ultrafilters in the lattice of the idempotents.  
    Furthermore, as it can be seen in \cite[Theorem IX.6.1]{Jacobson:StructureOfRings:1956} and  \cite[Theorem I]{Keimel}, the unital restriction on $A$   may be  relaxed in Theorem~\ref{t: JacobsonKeimel}, replacing  ${\mathscr L}(\hat{A}),$ by the algebra of the locally constant  functions $f:\hat{A}\to \kappa$ with compact ${\rm supp}\, f.$ In this case $\hat{A}$ is a locally compact totally disconnected Hausdorff space. 
\end{remark}

\begin{remark}
    A unital homomorphism $\psi : A\to B,$ where $A$ and $B$ are commutative $\kappa$-algebras generated by  idempotents induces a map $\hat{\psi}: \hat{B}\ni \varphi \mapsto \varphi \circ \psi\in \hat{A},$ which is easily seen to be continuous. Moreover, if  $\psi $ is surjective, then  $\hat{\psi}$ is injective. In addition, $\psi $ is an isomorphism if and only if $\hat{\psi}$ is a homeomorphism. 
\end{remark}

Note also  the following easy fact.

\begin{lemma}
    Let $X$ be compact totally disconnected Hausdorff space. Then the evaluation map $X \ni x \mapsto \hat{x}\in \widehat{{\mathscr L}(X)},$ $\hat{x}(f)= f(x),$ $ (f\in {\mathscr L}(X)),$ is continuous and injective. 
\end{lemma} 
\begin{proof} 
    If $x,y\in X$ and $x\neq y,$ then there exists a compact  open set $F$ containing $x$ which does not contain $y.$   Then $1_F\in {\mathscr L}(X)$ and $\hat{x}(1_F) = 1,$ whereas $\hat{y}(1_F) = 0,$ showing that our map is injective. Clearly,  the inverse image of $$Z(f, \alpha )=\{\varphi \in  \widehat{{\mathscr L}(X)} \, : \, \varphi (f)=\alpha\}$$ is $f\m(\alpha ) ,$ which is open in $X,$ proving that   $ x \mapsto \hat{x}$ is continuous. 
\end{proof}

Let us fix now a  set $E$ of idempotents which  generate our commutative unital algebra  $A.$ Then every $\varphi \in \hat{A}$ is completely determined by its values on the elements of  $E$. Write 
\[
  \zeta _{\varphi}  =\{ e\in E\, : \, \varphi (e)=1 \}\in 2^E, 
\]
where we identify $2^E$ with the set of all subsets of $E$.

\begin{proposition} 
    Taking the discrete topology on $\textbf{2}=\{0,1\}$ and the product topology on $\textbf{2}^E,$ the map 
    \[
        \hat{A} \to 2^E, \;\;\; \varphi \mapsto \zeta_{\varphi}
    \]
    is a homeomorphism onto its image.
\end{proposition} 
\begin{proof} 
Our map is evidently continuous and injective. Since $\hat{A}$ is compact and $2^E$ is Hausdorff, it must be a homeomorphism. 
\end{proof}

\section{Basics on partial projective representations and twisted partial group actions}\label{sec:Basics}

We proceed by recalling the concept of a partial projective $\sigma$-representation of a group and that of the twisted partial group algebra $\kappa_{par}^{\sigma}G$.

\begin{definition}[Definition 2.9 \cite{Dokuchaev-Jerez:twisted:2024}]\label{d: parproj} 
    A \textbf{partial projective $\sigma$-representation} of a group $G$ in a unital $\kappa$-algebra $R$ is a map $\Gamma : G \to R,$ for which there is a function $\sigma : G\times G \to \kappa,$ called \textbf{factor set}, such that the following postulates are satisfied for all $g,h \in G:$
    \begin{enumerate}[(i)]
        \item if $\sigma (g,h) = 0$, then $\Gamma (g^{-1}) \Gamma (g h)=0 \text{ and } \Gamma (g h) \Gamma(h^{-1}) = 0$;
        \item if $\Gamma(g) \Gamma(h) = 0$, then $\sigma(g, h)= 0$;
        \item $\Gamma (1_G) = 1_R$;
        \item $\Gamma (g^{-1}) \Gamma (g) \Gamma (h) = \sigma (g,h) \Gamma (g^{-1}) \Gamma (g h)$;
        \item $\Gamma (g) \Gamma (h) \Gamma (h^{-1}) = \sigma (g,h) \Gamma (gh) \Gamma (h^{-1})$.
    \end{enumerate}
    The set of all factor sets of $G$ is denoted by $pm(G)$.
\end{definition} 

\begin{remark} \label{r: sigma properties}
    It is important to note that the set $pm(G)$ is a commutative inverse monoid with the point-wise multiplication. Furthermore, observe that the relations in Definition~\ref{d: parproj} imply the following (see for example \cite{Dokuchaev-Jerez:twisted:2024} and \cite{DN} for more details):
    \begin{enumerate}[(i)]
        \item $\sigma(g,h) = 0 \Leftrightarrow \Gamma(g) \Gamma(h) = 0 \Leftrightarrow \Gamma (g^{-1}) \Gamma (g h)=0 \Leftrightarrow \Gamma (g h) \Gamma(h^{-1}) = 0$,
        \item $\sigma(1,1) = 1$,
        \item $\sigma(1,g) = \sigma(1, g^{-1}) = \sigma(g, 1) = \sigma(g^{-1}, 1) \in \{ 0, 1 \}$ for all $g \in G$,
        \item $\sigma(g, g^{-1}) = \sigma(g^{-1}, g)$ for all $g \in G$,
        \item $\sigma(g, h) = 0 \Leftrightarrow \sigma(h^{-1},g^{-1})=0$.
    \end{enumerate}
\end{remark}

\begin{definition}[Definition 2.10 \cite{Dokuchaev-Jerez:twisted:2024}]
    Let $\sigma \in pm(G).$  The \textbf{twisted partial group algebra} of $G$  is the  (universal)  $\kappa $-algebra $\kappa_{par}^{\sigma}G$ generated by the  symbols  $[g]^\sigma, \, (g\in G),$ subject to the relations:
    \begin{enumerate}[(i)]
        \item $\sigma(g,h)=0 \Rightarrow [g^{-1}]^\sigma [gh]^\sigma=[gh]^\sigma  [h ^{-1}]^\sigma =0$, 
        \item $[g^{-1}]^\sigma [g]^\sigma [h]^\sigma = \sigma (g,h) [g^{-1}]^\sigma [gh]^\sigma$, 
        \item $[g]^\sigma [h]^\sigma [h^{-1}]^\sigma = \sigma (g,h) [gh]^\sigma  [h^{-1}]^\sigma$,
        \item $[g]^\sigma [1_G]^\sigma = [1_G]^\sigma [g]^\sigma=[g]^\sigma$, 
    \end{enumerate}
    for all $g,h \in G.$
\end{definition}

The partial group algebra controls the partial projective $\sigma$-representations, as stated in the following fact:

\begin{proposition}[Proposition 2.11 \cite{Dokuchaev-Jerez:twisted:2024}]
    Let $\sigma \in pm(G).$ Then, the map $[\_]^{\sigma}: G \to \kappa_{par}^{\sigma}G$ such that $g \mapsto [g]^{\sigma}$ is a partial $\sigma$-representation and $\kappa_{par}^{\sigma}G$ is universal in the following sense: for any partial $\sigma$-representation $\Gamma$ of $G$ into a unital $\kappa$-algebra $R$ there exists a unique $\kappa$-algebra homomorphism $\tilde{\Gamma} : \kpar ^\sigma G \to R$ such that the diagram
    \begin{equation}\label{univProperty} 
        \begin{tikzcd}
            G & R \\
            {\kappa_{par}^\sigma G}
            \arrow["\Gamma", from=1-1, to=1-2]
            \arrow["{[\_]^\sigma}"', from=1-1, to=2-1]
            \arrow["{\tilde{\Gamma}}"', curve={height=6pt}, dashed, from=2-1, to=1-2]
        \end{tikzcd}
    \end{equation}
    commutes.
    Conversely, for any $\kappa$-algebra homomorphism $\tilde{\Gamma} : \kpar ^\sigma G \to R$ gives rise to a partial $\sigma$-representation $\Gamma :G \to R$ such that \eqref{univProperty} commutes.
\end{proposition}

In order to recall the description of $\kappa_{par}^{\sigma}G$ as a twisted partial crossed product we need the following definitions:

\begin{definition} 
    A \textbf{unital partial action} of  a group $G$ on a  $\kappa$-algebra $A$  is a pair $\theta=(\{A_g\}_{g\in G}, \{\theta_g\}_{g\in G})$, such that for every $g\in G,$ $A_g$ is  an ideal of $A,$ which is a unital algebra with unity element  $1_g,$ $\theta_g \colon A_{g^{-1}}\to A_g$ is a $\kappa$-algebra isomorphism, satisfying the following postulates for all  $g,h,t\in G:$
    \begin{enumerate} [(i)]
        \item $A_1=A$ and $\theta_1 = \operatorname{id}_{A}$,
        \item $\theta_g(A_{g^{-1}}A_h)=A_g A_{gh},$ 
        \item $\theta_g \circ \0 _h(a)=\theta_{gh}(a)$ for any $a\in A_{h\m} A_{(gh)^{-1}}.$
    \end{enumerate}
\end{definition}

\begin{definition}  
    Let $\theta$ be a unital  partial action of a group $G$ on a $\kappa$-algebra $A$. By a \textbf{$\kappa$-valued twist} of $\theta$ we mean a function $\sigma :  G \times G \to \kappa$ such that: 
    \begin{enumerate}[(i)]
        \item $\sigma(g,h)=0$ if, and only if, $1_g 1_{gh} = 0$;
        \item if $1_g 1_{gh}1_{ght} \neq 0$, then $\sigma(g,h) \sigma(gh,t) = \sigma(g,ht) \sigma(h,t)$.
    \end{enumerate}
\end{definition}

\begin{proposition}[Corollary 3.8 \cite{Dokuchaev-Jerez:twisted:2024}]
  The semigroup $pm(G)$ coincides the set of all $\kappa$-valued partial twists for $G.$ 
\end{proposition}

Given a unital partial action $\theta$ of $G$ on $A$ and a $\kappa$-valued twist $\sigma$ of $\theta$ the \textbf{crossed product} $A \rtimes_{(\theta,\sigma)} G$ is the direct sum:
 $$
 \bigoplus_{g \in G} A_g \de_g,
 $$ 
 in which the $\de_g$'s are placeholder symbols. The multiplication is defined by the rule: 
 \[
     (a_g \delta_g) \cdot (b_h \delta_h) = \sigma(g,h) a_g {\theta}_g (1_{g^{-1}} b_h)\delta_{gh}.
 \]
 By \cite[Theorem 2.4]{DES1}, $A \rtimes_{(\theta, \sigma)} G$ is an associative unital $\kappa$-algebra.

For all \( g \in G \), we define 
\[
    e^\sigma_g := \left\{\begin{matrix}
        \sigma(g, g^{-1})^{-1} [g]^{\sigma} [g^{-1}]^{\sigma} & \text{ if } \sigma(g, g^{-1}) \neq 0, \\ 
            0 & \text{ otherwise,}
        \end{matrix}\right.
\]
Let \( B^\sigma \) be the subalgebra of \( \kappa_{par}^\sigma G \) generated by the set \( \{ e_g^\sigma \}_{g \in G} \). By \cite[Theorem 3.11]{Dokuchaev-Jerez:twisted:2024}, the subalgebra $B^{\sigma}$ is commutative, each $e_{g}^{\sigma}$ is idempotent and there exists a unital partial action \(\theta^{\sigma}\) of \( G \) on \( B^\sigma \) with twist \( \sigma \) such that \( B^\sigma \rtimes_{\sigma} G \cong \kappa_{par}^\sigma G \), where \(B^\sigma \rtimes_{\sigma} G\) is a simplified notation for \( B^\sigma \rtimes_{(\theta^{\sigma}, \sigma)} G\). Explicitly, this partial group action is given as follows:
\[
    \theta^{\sigma}:=(G, B^{\sigma}, \{ B^{\sigma}_{g} \}, \{ \theta_{g}^{\sigma}\}),
\]
such that $B^{\sigma}_{g} := e_{g}^{\sigma}B^{\sigma}$ and 
\[
    \theta_{g}^{\sigma}(w) := \left\{\begin{matrix}
        \sigma(g^{-1}, g)^{-1} [g]^{\sigma} w [g^{-1}]^{\sigma} & \text{ if } \sigma(g, g^{-1}) \neq 0, \\ 
            0 & \text{ otherwise,}
        \end{matrix}\right.
\]
This shows the fundamental role that \( B^\sigma \) plays in the theory of projective partial representations. Therefore, it is natural to consider questions regarding \( B^\sigma \) and the natural partial group action of \( G \) on it.

\begin{remark}
    Observe that in the particular case where $\sigma = 1_{pm(G)}$, we obtain that $\kappa_{par}^{\sigma}G = \kappa_{par}G$ and $\kappa_{par}G \cong B \rtimes G$, where $B := B^{\sigma}$. Thus, $B$ is the commutative subalgebra of $\kappa_{par}G$ generated by the set of idempotent elements $\{ e_{g} := [g][g^{-1}] \}_{g \in G}$.
\end{remark}

\begin{remark} \label{r: kappa generators of kparsG}
    A simple consequence of \cite[Theorem 3.11]{Dokuchaev-Jerez:twisted:2024} is that  
    \[
        \left\{ \left( \prod_{h \in U} e^{\sigma}_{h} \right) [g]^{\sigma} : U \subseteq G, \, |U| < \infty, \, 1, g \in U \right\}
    \]
    generates $\kappa_{par}^{\sigma}G$ as a $\kappa$-vector space.
\end{remark}

\section{Spectral partial action}\label{sec:Spectral}

In order to consider the spectrum of \(B^{\sigma}\) as a subspace of the spectrum of \(B\), it will be useful to describe \(B^{\sigma}\) as a quotient algebra of \(B\). 

\begin{remark} \label{r: B is free}
    From \cite[Theorem 10.9]{E6}, it is easy to verify that the algebra \( B \) is a free commutative algebra generated by the set of idempotents \( \{ e_{g} : g \in G \} \). 
\end{remark}

 \begin{proposition} \label{l: B Isigma quotient}
     Let $\mathcal{I}_\sigma$ be the ideal of $B$ generated by the elements
     \begin{enumerate}[(i)]
         \item $e_s e_{sg} e_{sgh}$ if $\sigma(g,h)=0$;
         \item $e_s e_{sg} e_{sgh} e_{sght}$ if $\sigma(g,h) \sigma(gh,t) \neq \sigma(g, ht) \sigma(h,t)$. 
     \end{enumerate}
     Then, $B/\mathcal{I}_\sigma \cong B^\sigma$. 
 \end{proposition}
 \begin{proof}
     Let $\overline{B} := B/ \mathcal{I}_{\sigma}$. By Remark~\ref{r: B is free} there exists a surjective morphism of algebras \( f_0: B \to B^{\sigma} \) such that \( e_{g} \in B \mapsto e^{\sigma}_{g} \in B^{\sigma} \). It is easy to see that $\mathcal{I}_{\sigma} \subseteq \ker f_0$, therefore there exists a morphism of algebras $f: \overline{B} \to B^{\sigma}$ such that $f(\overline{e_{g}}) = e^{\sigma}_{g}$, where $\overline{e_{g}}$ denotes the class of $e_{g}$ in $\overline{B}$. Recall that $B$ is a left $\kappa_{par}G$-module, moreover the ideal $\mathcal{I}_{\sigma}$ is also a left $\kappa_{par}G$-module, therefore $\overline{B}$ is a $\kappa_{par}G$-module. Let $\alpha=( \{ D_{g} \}, \{ \alpha_{g} \})$ be the unital partial group action of $G$ on $\overline{B}$ induced by its $\kappa_{par}G$-module structure. Therefore, $D_{g}:= e_{g} \cdot \overline{B} = \overline{e_{g} B}$. Note that $\sigma$ is a twist for $\alpha$. Indeed, it is clear that $\overline{e_{g} e_{gh}}=0$ if $\sigma(g,h)=0$, and $\overline{e_{g} e_{gh} e_{ght}}=0$ if $\sigma(g,h) \sigma(gh,t) \neq \sigma(g, ht) \sigma(h,t)$. Now observe that
     \[
         \overline{e_{g} e_{gh}} = 0 \Rightarrow e^{\sigma}_{g}e^{\sigma}_{gh} = f(\overline{e_{g} e_{gh}})=0 \Rightarrow \sigma(g,h)=0.
     \]
     Thus, $\sigma$ is a twist for $\alpha$. Hence, there exists an homomorphism of algebra $\kappa_{par}^{\sigma}G \to \overline{B} \rtimes_{\alpha, \sigma} G$ such that $[g]^{\sigma} \mapsto \overline{e_{g}}\delta_{g}$, in particular the restriction of this homomorphism to $B^{\sigma}$ is exactly the inverse map of $f$.
 \end{proof}

Recall, that $B^\sigma$ is a $\kappa_{par}^{\sigma}G$-module, and the respective action is given by  
\[
    [g]^{\sigma} \cdot w = [g]^{\sigma} w [g^{-1}]^{\sigma}.
\]

\begin{definition} 
    A \textbf{topological partial action} of  a group $G$ on a to\-po\-lo\-gi\-cal space $X$  is a pair $\theta=(\{X_g\}_{g\in G}, \{\theta_g\}_{g\in G})$, such that for every $g\in G,$ $X_g$ is  an open subset $X,$          $\theta_g \colon X_{g^{-1}}\to X_g$ is a homeomorphism, such that  the following postulates are satisfied for all  $g,h,t\in G:$
    \begin{enumerate} [(i)]
        \item $X_1=X$ and $\theta_1 = \operatorname{id}_{X}$,
        \item $\theta_g(X_{g^{-1}}\cap X_h)=X_g \cap X_{gh},$ 
        \item $\theta_g \circ \0 _h(x)=\theta_{gh}(x)$ for any $x\in X_{h\m} \cap X_{(gh)^{-1}}.$
    \end{enumerate}
\end{definition}

Let $X = \operatorname{Spec} B^{\sigma}$. Then, $X \cong \hat{B}^\sigma := \{ f: B^\sigma \to \kappa : f \text{ is a homomorphism} \}$. Moreover, the sets
\[
    D(u):= \{ \mathfrak{p} \in X : u \notin \mathfrak{p} \} \cong \{ f \in \hat{B}^{\sigma} : f(u)=1 \},
\]
determines a basis of the topology of $X$, where $u$ is an idempotent element of $B^{\sigma}$.

We can define a partial action of $G$ on the spectrum of $B^\sigma$ as follows
\begin{enumerate}[(i)]
    \item $\hat{D}_g := D(e_g)$, for all $g \in G$;
    \item $\hat{\theta}_{g}(f) := f \circ \Theta_{g^{-1}}$ for all $f \in \hat{D}_{g^{-1}}$,
\end{enumerate}
where $\Theta_{g}(w):= \theta_{g}(e_{g^{-1}}^{\sigma}w)$ for all $w \in B^{\sigma}$.

\begin{remark}
    Due to the decomposition $B^{\sigma}= B^{\sigma}_{g} \oplus (1 - e_{g})B^{\sigma}$ we have that the set of homomorphisms $\hat{D}_{g}$ corresponds to the maps $f \in \hat{B}^{\sigma}$ such that $1 - e_{g}^{\sigma} \in \ker f$. So that, $\hat{D}_{g}$ can be seen as the set of all ring homomorphisms from $B_{g}^{\sigma}$ to $\kappa$. Moreover, note that if $B_{g}^{\sigma} = \{ 0 \}$ then such a set is empty.
\end{remark}

\begin{remark} \label{r: spectral partial group action}
    The above construction can be easily extended to the case of a partial action \(\alpha\) of a group $G$ on a unital algebra $A$ generated by idempotents.
    In this case, we can define a topological partial action \(\hat{\alpha}:= \{ \hat{\alpha}_{g}: \hat{D}_{g^{-1}} \to \hat{D}_{g} \}_{g \in G}\) of \(G\) on \(\operatorname{Spec} A\), with clopen \(\hat{D}_{g}\), such that:
    \begin{enumerate}[(i)]
        \item $\hat{D}_g := D(1_g)$, for all $g \in G$;
        \item $\hat{\alpha}_{g}(f) := f \circ \Theta_{g^{-1}}$ for all $f \in \hat{D}_{g^{-1}}$,
    \end{enumerate}
    where $\Theta_{g}(a):= \alpha_{g}(1_{g^{-1}}a)$ for all $a \in A$.
\end{remark}

\begin{remark}
    If $\sigma = 1$ we set $\hat{B}:= \{ f : B \to \kappa : f \text{ is homomorphism} \}$. Furthermore, by Remark~\ref{r: B is free} we can identify $\hat{B}$ with $\Omega := \{ C \subseteq G : 1 \in C \}$, via $\xi \in \Omega \mapsto f_\xi \in \hat{B}$, such that $f_{\xi}(e_g) = 1$ if, and only if, $g \in \xi$. This is a homeomorphism, where the topology of $\Omega$ is the product topology given by the identification $\Omega = \prod_{g \in G \setminus \{ 1_G \}} \{ 0,1 \}$.
\end{remark}

\begin{remark}
    Notice that if $G$ is a finite group the spectrum of $\hat{B}^{\sigma}$ has the discrete topology.
    \label{r: Omega sigma is discrete when G is finite}
\end{remark}

\begin{definition}
    We define
    \[
        \Omega_\sigma := \left\{ \xi \in \Omega : \mathcal{I}_\sigma \subseteq \ker f_\xi \right\}.
    \]
\end{definition}
Note that we can identify $\Omega_\sigma$ with $V(\mathcal{I}_\sigma)$. Therefore, $\Omega_\sigma$ is closed in $\hat{B}$. Moreover, we have the following easy lemma:
\begin{lemma}
    The map $\xi \in \Omega_\sigma \mapsto \bar{f_\xi} \in \hat{B}^\sigma$ is a homeomorphism, where $\bar{f_{\xi}}: B^{\sigma} \to \kappa$ is the homomorphism induced by $f_{\xi}$.
    \label{l: Omega sigma is hat B sigma}
\end{lemma}

Due to Lemma~\ref{l: Omega sigma is hat B sigma} we can describe the partial group action $\hat{\theta}$ of $G$ on $\hat{B}^{\sigma}$ as a partial group action on $\Omega_{\sigma}$ as follows:

\begin{enumerate}[(i)]
    \item $\mathcal{D}_{g}^{\sigma}:= \{ \xi \in \Omega_{\sigma} : g \in \xi \}$
    \item $\vartheta_{g}: \mathcal{D}^{\sigma}_{g^{-1}} \to \mathcal{D}^{\sigma}_{g}$ such that $\vartheta_g(\xi) := g \xi$.
\end{enumerate}
Indeed, it is easy to see that $\mathcal{D}^{\sigma}_{g}$ corresponds to $\hat{B}^{\sigma}_{g}$ and that $\hat{\theta}_{g}(\bar{f}_{\xi}) = \bar{f}_{g \xi}$ for all $\bar{f}_{\xi} \in \hat{D}_{g^{-1}}$.

\begin{definition}
    We define the set $\sigma$-prohibitions of type $\mathrm{I}$ by
    \[
        \mathcal{P}_\sigma^{\mathrm{I}} := \left\{ \{ h, hg, hgs \} : \sigma(g,s) = 0 \right\},
    \]
    the set $\sigma$-prohibitions of type $\mathrm{II}$  by
    \[
        \mathcal{P}_\sigma^{\mathrm{II}} := \left\{ \{ h, hg, hgs, hgst \} : \sigma(g,st)\sigma(s,t) \neq \sigma(g,s) \sigma(gs,t) \right\},
    \]
    and the set of $\sigma$-prohibitions by
    \[
       \mathcal{P}_\sigma := \mathcal{P}_\sigma^{\mathrm{I}} \cup \mathcal{P}_\sigma^{\mathrm{II}}.
    \]
\end{definition}

It is easy to see that for any $\xi \in \Omega$ we have that $f_\xi(e_{g_1}e_{g_2} \ldots e_{g_n}) = 0$ if, and only if, $\{ g_1, \ldots, g_n \} \nsubseteq \xi.$ Thus, by Proposition~\ref{l: B Isigma quotient} we have that
\begin{align*}
    \Omega_{\sigma} 
    &= \{ \xi \in \Omega : \mathcal{I}_\sigma \subseteq \ker f_\xi \} \\
    &= \{ \xi \in \Omega : f_{\xi}(w) = 0 \text{ for all } w \in \mathcal{I}_{\sigma} \} \\
    &= \{ \xi \in \Omega : f_{\xi}(e_{s} e_{sg} e_{sgh}) = 0 \text{ if } \sigma(g,h) = 0 \\ 
    & \hspace{1cm}\text{ and }  f_{\xi}(e_{s} e_{sg} e_{sgh} e_{sght}) = 0 \text{ if }  \sigma(g, h) \sigma(gh, t) \neq \sigma(g, ht) \sigma(h,t) \} \\
    &= \{ \xi \in \Omega : V \nsubseteq \xi \text{ for all } V \in \mathcal{P}_{\sigma} \}.
\end{align*}
Therefore, we obtain the following characterization of $\Omega_\sigma$ using $\sigma$-prohibitions
\begin{equation}
    \Omega_\sigma = \{ \xi \in \Omega : V \nsubseteq \xi, \, \forall V \in \mathcal{P}_\sigma \}.
    \label{eq: Omega sigma characterization by prohibitions}
\end{equation}

A direct consequence of \eqref{eq: Omega sigma characterization by prohibitions} is the following lemma:

\begin{lemma}
    Let $\xi \in \Omega_\sigma$. Then, $\xi' \in \Omega_\sigma$ for all $\xi' \in \Omega$ such that $\xi' \subseteq \xi$.
    \label{l: Omega sigma contains subset elements}
\end{lemma}

\begin{remark} \label{r: zeros of B sigma induced by Omega sigma}
    Recall that 
    \begin{enumerate}[(i)]
        \item $\sigma(g,h) = 0$ if, and only if, $e^{\sigma}_{g}e_{gh}^{\sigma} = 0$
        \item if $\sigma(x,y) \sigma(xy, z) \neq \sigma(x, yz) \sigma(y, z)$, then $e_{h}^{\sigma} e_{hx}^{\sigma} e_{hxy}^{\sigma} e_{hxyz}^{\sigma} = 0$ for all $h \in G$.
    \end{enumerate}
    Hence, it is easy to see that if $V \in \mathcal{P}_{\sigma}$, then $\prod_{g \in V} e^{\sigma}_{g} = 0$.
    Therefore, if $\xi \in \Omega$ is such that $\xi \notin \Omega_{\sigma}$, $|\xi| < \infty$, then $\prod_{g \in \xi} e^{\sigma}_{g} = 0$. 
\end{remark}

Consider the partial order in $\Omega$ and $\Omega_\sigma$ determined by the natural order of sets by inclusion. Note that Lemma~\ref{l: Omega sigma contains subset elements} implies that $\Omega_{\sigma}$ is an order ideal in $\Omega$.

\begin{proposition}
    For each $\xi \in \Omega_\sigma$  there exists $\hat{\xi} \in \Omega_{\sigma}$ such that $\xi \subseteq \hat{\xi}$ and $\hat{\xi}$ is maximal in $\Omega_{\sigma}$.
    \label{p: Omega sigma has maximal elements}
\end{proposition}
\begin{proof}
    Let $\xi \in \Omega_{\sigma}$ and let $\Omega_{\sigma}(\xi):= \{ U \in \Omega_{\sigma} : \xi \subseteq U\}$.
    Let $\{ \xi_i \}_{i \in I}$ be a chain in $\Omega_{\sigma}(\xi)$. Suppose that there exists $V = \{ h_1, h_2, h_3, h_4 \} \in \mathcal{P}_\sigma$ such that $V \subseteq \xi_0 := \bigcup_{i \in I} \xi_i$. Then, since $\{ \xi_i \}$ is a chain there exists $i_0 \in I$, such that $V \subseteq \xi_{i_0}$, but this contradicts the fact that $\xi_{i_0} \in \Omega_\sigma$. Thus, $\xi_0 \in \Omega_\sigma$ and by Zorn's lemma there exists a maximal element $\hat{\xi}$ in $\Omega_{\sigma}(\xi) $. Finally, it is easy to see that $\hat{\xi}$ is also maximal in $\Omega_{\sigma}$.
\end{proof}

\begin{lemma}
    Let $u \in B^{\sigma}$ be idempotent and consider the open set $D(u)$ of $\operatorname{Spec} B^{\sigma}$. Then,
    \begin{enumerate}[(i)]
        \item For each $\xi \in D(u)$ there exists $\xi' \in D(u)$ such that $\xi' \subseteq \xi$, $\xi'$ is finite and minimal in $D(u)$,
        \item For each $\xi \in D(u)$ there exists $\hat{\xi} \in D(u)$ such that $\xi \subseteq \hat{\xi}$ and $\hat{\xi}$ is maximal in $D(u)$.
    \end{enumerate}
    \label{l: minimal and maximal elements of Du}
\end{lemma}
\begin{proof}
    Let $u$ an idempotent element of $B^{\sigma}$. There exists a finite subset $W$ of $G$ such that
    \[
        u = \sum_{I \subseteq W} \alpha_I \left( \prod_{i \in I} e_{g_{i}}^{\sigma} \right), \, \, \text{ where } \alpha_I \in \kappa.
    \]
    For $(i)$ it is enough to show that for all $\xi \in D(u)$ there exists $\xi' \subseteq \xi$, such that $|\xi'| < \infty$ and $\overline{f_{\xi'}}(u) = 1$. Let $\xi \in D(u)$ and set $\xi' = W \cap \xi$. Notice that by Lemma~\ref{l: Omega sigma contains subset elements} we have that $\xi' \in \Omega_\sigma$, and on the other hand, it is clear that $f_{\xi'}(e_g) = f_{\xi}(e_g)$ for all $g \in W$. Therefore,
    \begin{align*}
        1 = \overline{f_{\xi}}(u) 
        &= \sum_{I \subseteq W} \alpha_I \left( \prod_{i \in I} \overline{f_{\xi}}(e_{g_{i}}^{\sigma}) \right) 
        = \sum_{I \subseteq W} \alpha_I \left( \prod_{i \in I} f_{\xi}(e_{g_{i}}) \right) \\
        &= \sum_{I \subseteq W} \alpha_I \left( \prod_{i \in I} f_{\xi'}(e_{g_{i}}) \right) 
        = \overline{f_{\xi'}}(u).
    \end{align*}
    For $(ii)$, let $\xi \in D(u)$ and let $\{ \xi_i \}$ be a chain in $D(u)$ such that $\xi \subseteq \xi_{i}$ for all $i$, set $\xi_0 = \cup \xi_i$, by the proof of Proposition~\ref{p: Omega sigma has maximal elements} we know that $\xi_0 \in \Omega_\sigma$. Since $W \cap \xi_0$ is finite, there exists $i_0$ such that $W \cap \xi_0 \subseteq \xi_{i_0}$. Then, $f_{\xi_{i_0}}(e_{g}) = f_{\xi_0}(e_{g})$ for all $g \in W$. Indeed, if $f_{\xi_0}(e_g) = 1,$ then $g \in \xi_0$, thus $g \in \xi_0 \cap W \subseteq \xi_{i_0}$, whence $f_{\xi_{i_0}}(g) = 1$; if $f_{\xi_0}(g) = 0$, then $g \notin \xi_0$, thus $g \notin \xi_{i_0}$ since $\xi_{i_0} \subseteq \xi_0$. Thus, $f_{\xi_{i_0}}(g) = 0$. Form this we obtain that
    \[
        \overline{f_{\xi_0}}(u) = \sum_{I \subseteq W} \alpha_I \left( \prod_{i \in I} f_{\xi_0}(e_{g_{i}}) \right) 
        = \sum_{I \subseteq W} \alpha_I \left( \prod_{i \in I} f_{\xi_{i_0}}(e_{g_{i}}) \right) 
        = \overline{f_{\xi_{i_0}}}(u)=1.
    \]
    Then, $\xi_0 \in D(u)$. Finally, the Zorn's lemma give us the desired result.
\end{proof}

\section{The \texorpdfstring{$\kappa$}{TEXT}-cancellative inverse monoid \texorpdfstring{$\mathcal{S}^{\sigma}(G)$}{TEXT}}\label{sec:CancellativeMon}

Besides using $\sigma$-prohibitions to characterize $B^{\sigma}$, we can provide another description of $\kappa_{par}^{\sigma}G$ by utilizing a $\kappa$-cancellative monoid determined by the $\sigma$-prohibitions.

\[
    \mathcal{S}^{\sigma}(G):= \{ (k, U, g) \in \kappa \times \Omega_\sigma \times G : g \in U, \, k \neq 0, |U| < \infty \} \cup \{ 0 \}.
\]
We have an obvious action of $\kappa$ on $\mathcal{S}^{\sigma}(G)$, defined as
\begin{align}
    r \cdot 0 &:= 0 \text{ for all } r \in \kappa; \\ 
    r \cdot (k, U, g) &:= (rk, U, g)  \text{ for all } r \in \kappa, \, r \neq 0; \\
    0 \cdot (k, U,g) &:= 0.
\end{align}
For the sake of simplicity we will denote $(k, U, g)$ by $k(U,g)$ and $1(U,g)$ by $(U,g)$. Therefore, the action of $\kappa$ is denoted by
\begin{align*}
    r \cdot (U, g) &= r(U, g)  \text{ for all } r \in \kappa, \, r \neq 0; \\
    0 \cdot (U,g) &= 0(U,g)=0.
\end{align*}
For the case of the trivial factor set $1$ we have that $\Omega_1 = \Omega$. Consequently, $\mathcal{S}^{\sigma}(G) \subseteq \mathcal{S}^{1}(G)$. Moreover, the action of $\kappa$ on $\mathcal{S}^{\sigma}(G)$ is just the restriction of the action of $\kappa$ on $\mathcal{S}^{1}(G)$ to $\mathcal{S}^{\sigma}(G)$. In particular, we obtain that $0(U,g) = 0$ for all $U \in \Omega$ and $g \in U$. 

Let $\chi$ be the characteristic function defined by
    \[
        \chi(U) := \left\{\begin{matrix}
            1 & \text{ if } U \in \Omega_\sigma, \\ 
            0 & \text{ otherwise.}
        \end{matrix}\right.
    \]
An easy fact about $\chi$ is its invariance by left translations, i.e., $\chi(U) = \chi(g^{^{-1}}U)$ for all $U \in \Omega_{\sigma}$, $|U| < \infty$, and $g \in U$.
This clearly follows the definition of $\sigma$-prohibitions.
Define the product $\star$ in $\mathcal{S}^{\sigma}(G)$ as follows
\begin{equation} \label{eq: product in S sigma}
    (k(U,g)) \star (r(V,h)) := \sigma(g,h) \chi(U \cup gV) kr (U \cup gV, gh) 
\end{equation}
and
\[
    0 \star z = z \star 0 := 0, \text{ for all } z \in \mathcal{S}^{\sigma}(G).
\]
\begin{lemma}
    $\chi(U \cup gV) \chi(U \cup gV \cup ghS) = \chi(V \cup hS) \chi(U \cup gV \cup ghS)$, for all $U, V, S \in \Omega_\sigma$, $g \in U$ and $h \in V$.
    \label{l: chi associative}
\end{lemma}
\begin{proof}
    It is clear that if $\chi(U \cup gV \cup ghS) = 0$ the equality holds. Now, if $\chi(U \cup gV \cup ghS) = 1$, by Lemma~\ref{l: Omega sigma contains subset elements} we have that $\chi(U \cup gV) = 1$ and $\chi(\{ 1_G \} \cup gV \cup ghS) = 1$. Thus, $\{ 1_G \} \cup gV \cup ghS \in \Omega_{\sigma}$, whence $\{ g^{-1} \} \cup V \cup hS = \vartheta_{g^{-1}}(\{ 1_G \} \cup gV \cup ghS) \in \Omega_{\sigma}$. Therefore, by Lemma~\ref{l: Omega sigma contains subset elements} we conclude that $\chi(V \cup hS) = 1$.
\end{proof}

We recall from \cite[Definition 2]{DN} that a \textbf{$\kappa$-semigroup} is a semigroup $T$ with a zero element $0_{T}$ and a map $\kappa \times T \to T$ satisfying the following properties for all $r,t \in \kappa$ and all $x,y \in T$: 
\begin{itemize}
    \item $r(t x) = (rt)x$,
    \item $r(xy) = (r x ) y = x (r y)$,
    \item $1_{\kappa} x =x$,
    \item $0_{\kappa} x = 0_{T}$.
\end{itemize}
A $\kappa$-semigroup $T$ is called \textbf{$\kappa$-cancellative} if additionally we have that: 
\[
    r x = t x \Longrightarrow r = t,
\]
for all $x \neq 0$.

\begin{proposition}\label{prop:Cancellative}
    $\mathcal{S}^{\sigma}(G)$ is a $\kappa$-cancellative inverse monoid.
\end{proposition}
\begin{proof}
    First we have to verify that the product of $\mathcal{S}^{\sigma}(G)$ is associative. Let $(U,g)$, $(V, h)$ and $(S, s)$ in $\mathcal{S}^{\sigma}(G)$, observe that if $\chi(U \cup gV \cup ghS) = 0$, then immediately we have that
    \[
        (U,g) \Big( (V,h)  (S,s) \Big) = 0 = \Big( (U,g) (V,h) \Big) (S,s).
    \]
    Now, assume that $\chi(U \cup gV \cup ghS) = 1$. Hence, by Lemma~\ref{l: Omega sigma contains subset elements}, we have that $\{ 1, g, gh, ghs \} \in \Omega_{\sigma}$ and by Lemma~\ref{l: Omega sigma is hat B sigma} we conclude that $e_{g}^{\sigma} e_{gh}^{\sigma}e_{ghs}^{\sigma} \neq 0$ whence $\sigma(g,h) \sigma(gh, s) = \sigma(g, hs) \sigma(h,s)$. Thus, we get
    \begin{align*}
        (U,g) \Big( (V,h)  (S,s) \Big)
        &= (U,g) \Big( \sigma(h,s) \chi(V \cup hS) (V \cup hS,hs)\Big) \\
        &= \sigma(g, hs) \sigma(h,s) \chi(V \cup hS) \chi(U \cup gV \cup ghS) (U \cup gV \cup ghS, ghs) \\
        (\flat) 
        &= \sigma(h,s) \sigma(g, hs) \chi(U \cup gV) \chi(U \cup gV \cup ghS) (U \cup gV \cup ghS, ghs) \\
        &= \sigma(g,h) \sigma(gh, s) \chi(U \cup gV) \chi(U \cup gV \cup ghS) (U \cup gV \cup ghS, ghs) \\
        &= \sigma(g,h) \chi(U \cup gV) (U \cup gV, gh) (S,s) \\
        &= \Big( (U,g) (V,h) \Big) (S,s),
    \end{align*}
    where $(\flat)$ holds by Lemma~\ref{l: chi associative}. It is clear that $(\{ 1 \}, 1)$ is the unit of $\mathcal{S}^{\sigma}(G)$. Now, it is easy to verify that for an element $k(U,g)$ of $\mathcal{S}^{\sigma}(G)$ we have that its inverse (in the sense of regular semigroups) is $\sigma(g,g^{-1})^{-1}k^{-1}(g^{-1}U, g^{-1})$. Indeed,
    \begin{align*}
       \sigma(g,g^{-1})^{-1} k(U,g) \star k^{-1}(g^{-1}U, g^{-1}) \star  k(U,g)
        & = (U, 1_{G}) \star k(U,g) \\ 
        & = \sigma(1, g)k(U, g) \\ 
        &= k(U,g),
    \end{align*}
    where the last equality holds since $\sigma(1, g) = 1$ by Remark~\ref{r: sigma properties}.
    Moreover, a non-zero element $k(U,g)$ is an idempotent if, and only if, $k=1$ and $g = 1$. Whence, we conclude that the idempotent elements of $\mathcal{S}^{\sigma}(G)$ commute with each other. Therefore, $\mathcal{S}^{\sigma}(G)$ is an inverse monoid. Finally, by construction it is clear the $\mathcal{S}^{\sigma}(G)$ is $\kappa$-cancellative.    
\end{proof}

Observe that $\mathcal{S}^{\sigma}(G)$ generalizes the Szendrei expansion of a group  see \cite{Sze} (which is isomorphic to the Exel's semigroup S(G) see \cite{KL} and to the Birget-Rhodes prefix expansion of $G$) in the following sense. If $\kappa$ is the field with two elements $\{ 0, 1 \}$ and $\sigma = 1$, then $S(G) = \mathcal{S}^{1}(G) \setminus \{ 0 \}$.

\begin{proposition}
    $\chi(\{ 1,g \}) = 0$ if, and only if, $\sigma(g,g^{-1})=0$
\end{proposition}

\begin{proof}
    By \cite[Corollary 6]{DN} we know that
    \[
        \sigma(g, g^{-1}) = 0 \Leftrightarrow \sigma(1, g) = 0 \Leftrightarrow \sigma(g, 1) = 0.
    \]
    Then, $\sigma(g, g^{-1}) = 0$ implies that $\{ 1, g \} \in \mathcal{P}_{\sigma}$. For the converse, assume that $\{ 1, g \} \in \mathcal{P}_{\sigma}$. Then, $\{ 1, g \} \in \mathcal{P}_{\sigma}^{\mathrm{I}}$ or $\{ 1, g \} \in \mathcal{P}_{\sigma}^{\mathrm{II}}$, then we have the following cases:
    \begin{enumerate}[(a)]
        \item if $\{ 1, g \} \in \mathcal{P}_{\sigma}^{\mathrm{I}}$, then there exists $h, x, y \in G$ such that $\{ 1, g \} = \{  h, hx, hxy \}$ and $\sigma(x,y)=0$. Then,
            \begin{enumerate}[(i)]
                \item if $h = 1$, then $\{ 1, x, xy \} = \{ 1, g \}$. Hence,
                    \begin{enumerate}[(1)]
                        \item if $x = 1$ and $y=g$, then $\sigma(1, g) = 0$;
                        \item if $x = g$ and $y=1$, then $\sigma(g, 1) = 0$;
                        \item if $x = g$ and $y=g^{-1}$, then $\sigma(g, g^{-1}) = 0$.
                    \end{enumerate}
                \item If $h = g$, then $\{ h, hx, hxy \} = \{ g, gx, gxy \}$. Hence,
                    \begin{enumerate}[(1)]
                        \item if $x = 1$ and $y=g^{-1}$, then $\sigma(1, g^{-1}) = 0$;
                        \item if $x = 1$ and $y=1$, then $\sigma(1, 1) = 0$, contradiction;
                        \item if $x = g^{-1}$ and $y=g$, then $\sigma(g^{-1}, g) = 0$;
                        \item if $x = g^{-1}$ and $y=1$, then $\sigma(g^{-1}, 1) = 0$.
                    \end{enumerate}
            Observe that each of these cases above implies that $\sigma(g, g^{-1})=0$.
            \end{enumerate}
        Recall also that by \cite[Proposition 5]{DN} we know that $\sigma(g,1) \in \{0, 1\}$ for all $g \in G$. The remaining cases are as follows:
        \item if $\{ 1, g \} \in \mathcal{P}_{\sigma}^{\mathrm{II}}$, then there exists $h, x, y, z \in G$ such that $\sigma(x,y)\sigma(xy, z) \neq \sigma(x, yz) \sigma(y,z)$ and $\{ 1, g \} = \{ h, hx, hxy, hxyz \}$, then
            \begin{enumerate}[(i)]
                \item if $h=1$, then $\{ 1, g \} = \{ 1, x, xy, xyz \}$. Hence,
                    \begin{enumerate}[(1)]
                        \item if $x = 1$, $y = 1$ and $z = g$,
                            then $\sigma(1, 1)\sigma(1,g)=\sigma(1,g)\sigma(1,g)$;
                        \item if $x = 1$, $y = g$ and $z = 1$,
                            then $\sigma(1, g)\sigma(g,1)=\sigma(1,g)\sigma(g,1)$;
                        \item if $x = 1$, $y = g$ and $z = g^{-1}$,
                            then $\sigma(1, g)\sigma(g,g^{-1})=\sigma(1,1)\sigma(g,g^{-1})$;
                            \item if $x = g$, $y = 1$ and $z = 1$,
                            then $\sigma(g, 1)\sigma(g,1)=\sigma(g,1)\sigma(1,1)$;
                        \item if $x = g$, $y = 1$ and $z = g^{-1}$,
                            then $\sigma(g, 1)\sigma(g,g^{-1})=\sigma(g,g^{-1})\sigma(1,g^{-1})$;
                        \item if $x = g$, $y = g^{-1}$ and $z = 1$,
                            then $\sigma(g, g^{-1})\sigma(1,1)=\sigma(g,g^{-1})\sigma(g^{-1}, 1)$;
                        \item if $x = g$, $y = g^{-1}$ and $z = g$,
                            then $\sigma(g, g^{-1})\sigma(1,g)=\sigma(g,1)\sigma(g^{-1},g)$;
                    \end{enumerate}
                    note that the last equality of item $(7)$ holds since $\sigma(g, g^{-1}) = \sigma(g^{-1},g)$ (see for example \cite[p. 261]{DN}).
                    Thus, in the above cases we obtain no prohibitions of type $II$;
                \item if $h = g$, then $\{ 1, g \} = \{ g, gx, gxy, gxyz \}$. Hence, 
                    \begin{enumerate}[(1)]
                        \item if $x = 1$, $y = 1$ and $z = g^{-1}$,
                            then $\sigma(1, 1)\sigma(1, g^{-1})=\sigma(1,g^{-1})\sigma(1, g^{-1})$;
                        \item if $x = 1$, $y = g^{-1}$ and $z = 1$,
                            then $\sigma(1, g^{-1})\sigma(g^{-1}, 1)=\sigma(1,g^{-1})\sigma(g^{-1},1)$;
                        \item if $x = 1$, $y = g^{-1}$ and $z = g$,
                            then $\sigma(1, g^{-1})\sigma(g^{-1},g)=\sigma(1,1)\sigma(g^{-1},g)$;
                        \item if $x = g^{-1}$, $y = 1$ and $z = g$,
                            then $\sigma(g^{-1}, 1)\sigma(g^{-1},g)=\sigma(g^{-1},g)\sigma(1, g)$;
                        \item if $x = g^{-1}$, $y = g$ and $z = 1$,
                            then $\sigma(g^{-1}, g)\sigma(1,1)=\sigma(g^{-1},g)\sigma(g,1)$;
                        \item if $x = g^{-1}$, $y = g$ and $z = g^{-1}$,
                            then $\sigma(g^{-1}, g)\sigma(1,g^{-1})=\sigma(g^{-1},1)\sigma(g,g^{-1})$.
                    \end{enumerate}
                Then, in subcase (ii) of (b) there are also no prohibitions.
            \end{enumerate}
    \end{enumerate}
    Therefore, the cases (a) and (b) yields to the desired conclusion.
\end{proof}

Let $\mathcal{A}_\sigma$ be the $\kappa$-vector space with basis $\{ (U, g) : U \in \Omega_{\sigma}, |U| < \infty, \text{ and } g \in U \}$. Observe that we have an injective map $\iota: \mathcal{S}^{\sigma}(G) \to \mathcal{A}_{\sigma}$ such that $\iota(r(U,g)):= r(U,g)$, which allows us to identify $\mathcal{S}^{\sigma}$ as a subset of $\mathcal{A}_{\sigma}$. Moreover, since the product $\star$ of $\mathcal{S}^{\sigma}(G)$ (see \eqref{eq: product in S sigma}) is defined in the elements of a basis of $\mathcal{A}_{\sigma}$ a direct computation shows that the multiplication of $\mathcal{S}^{\sigma}(G)$ endows $\mathcal{A}_{\sigma}$ with a $\kappa$-algebra structure.

\begin{lemma} \label{l: chi small computation}
    Let $g,h \in G$, then
    \[
       \chi(\{ 1, g \}) \chi(\{ \{ 1, g, h \} \}) =  \chi(\{ \{ 1, g, h \} \})
    \]
    
\end{lemma}
\begin{proof}
    Observe that if $\chi(\{ 1, g \}) = 0$ then $\chi(\{ \{ 1, g, h \} \}) = 0$ by Lemma~\ref{l: Omega sigma contains subset elements}, and if $\chi(\{ 1, g \}) = 1$ the equality trivially holds. 
\end{proof}

\begin{theorem}
    There exists an isomorphism of algebras $\kappa_{par}^{\sigma}G \to \mathcal{A}_{\sigma}$ such that $\phi([g]^{\sigma}) = \chi(\{ 1, g \})(\{ 1, g \}, g)$.
    \label{t: kparsG isomorphism}
\end{theorem}
\begin{proof}
    Consider the map
    \begin{align*}
        \pi: G & \to \mathcal{A}_{\sigma} \\
            g &\mapsto \chi(\{ 1, g \})\big(\{ 1, g \}, g \big).
    \end{align*}
    It is immediate that $\pi_1 = (\{ 1 \}, 1) = 1_{\mathcal{A}_{\sigma}}$. Let $g, h \in G$, first observe that by the left invariance of $\sigma$-prohibitions we obtain
    \begin{equation} \label{eq: left invariance chi formula}
        \chi(\{ 1, g, gh \}) = \chi(\{ 1, g^{-1}, h \}).
    \end{equation}
    Note that
    \begin{align*}
        \pi(g^{-1})& \pi(g) \pi(h) = \\
        &= \chi(\{ 1, g^{-1} \}) \chi(\{ 1, g \}) \chi(\{ 1, h \}) \bigl( \{ 1, g^{-1} \}, g^{-1} \bigr) \bigl( \{ 1, g \}, g \bigr) \bigl( \{ 1, h \}, h \bigr) \\
        &= \chi(\{ 1, g^{-1} \}) \chi(\{ 1, g \}) \chi(\{ 1, h \}) \chi(\{ \{ 1, g, gh \} \}) \sigma(g,h) \bigl( \{ 1, g^{-1} \}, g^{-1} \bigr) \bigl( \{ 1, g, gh \}, gh \bigr)  \\
        &= \chi(\{ 1, g^{-1} \}) \chi(\{ 1, g \}) \chi(\{ 1, h \}) \chi(\{ \{ 1, g, gh \} \}) \chi(\{ \{ 1, g^{-1}, h \} \})\sigma(g,h) \sigma(g^{-1}, gh) \bigl( \{ 1, g^{-1}, h \}, h \bigr) \\
        &= \chi(\{1, g, gh \}) \sigma(g,h) \sigma(g^{-1}, gh) \bigl( \{ 1, g^{-1}, h \}, h \bigr),
    \end{align*}
    where the last equality holds by Lemma~\ref{l: chi small computation} and equality \eqref{eq: left invariance chi formula}. On the other hand, using again Lemma~\ref{l: chi small computation} and equation \eqref{eq: left invariance chi formula} we get
    \begin{align*}
        \pi(g^{-1}) \pi(gh)
        &= \chi(\{ 1, g^{-1} \}) \chi(\{ 1, gh \}) \big( \{ 1, g^{-1} \}, g^{-1} \big) \big( \{ 1, g h \}, gh \big) \\
        &= \chi(\{ 1, g^{-1} \}) \chi(\{ 1, gh \}) \chi(\{ 1, g^{-1}, h \}) \sigma(g^{-1}, gh) \big( \{ 1, g^{-1}, h \}, h \big) \\
        &= \chi(\{ 1, g, gh \}) \sigma(g^{-1}, gh) \big( \{ 1, g^{-1}, h \}, h \big).
    \end{align*}
     Hence, we conclude that $\pi(g^{-1})\pi(g) \pi(h) = \sigma(g,h)\pi(g^{-1})\pi(gh)$. Analogously, one proves that $\pi(g) \pi(h) \pi(h^{-1}) = \sigma(g,h)\pi(gh) \pi(h^{-1})$. Thus, by the universal property of $\kappa_{par}^{\sigma}G$ there exists a morphism of algebras $\phi: \kappa_{par}^{\sigma}G \to \mathcal{A}_{\sigma}$ such that $\phi([g]^{\sigma}) = \chi(\{ 1, g \})(\{ 1, g \}, g)$.
    By a direct computation one verifies that
    \begin{equation*}
        \phi(e_{h_1}^{\sigma}e_{h_2}^{\sigma} \ldots e_{h_n}^{\sigma} ) = \chi(\{ 1, h_1, \ldots, h_n \}) \big( \{ 1, h_1, \ldots, h_n \}, 1 \big)
    \end{equation*}
    for all $h_1, \ldots, h_n \in G$. Thus, we have that
    \begin{equation} \label{eq: phi in basis of kparsG}
        \phi(e_{h_1}^{\sigma}e_{h_2}^{\sigma} \ldots e_{h_n}^{\sigma} [g]^{\sigma} ) = \chi(\{ 1, g, h_1, \ldots, h_n \}) \big( \{ 1, g, h_1, \ldots, h_n \}, g \big).
    \end{equation}
    Consider the $\kappa$-linear map
    \begin{align*}
        \phi': \mathcal{A}_{\sigma} &\to \kappa_{par}^{\sigma}G \\
                (U, g) &\mapsto (\prod_{g \in U} e^{\sigma}_{g})[g]^{\sigma}.
    \end{align*}
    By Remark~\ref{r: kappa generators of kparsG} we know that
    \[
        \left\{ \left( \prod_{h \in U} e^{\sigma}_{h} \right) [g]^{\sigma} : U \subseteq G, \, |U| < \infty, \, 1, g \in U \right\}
    \]
    generates $\kappa_{par}^{\sigma}G$ as $\kappa$-vector space. Moreover, if $\prod_{h \in U} e^{\sigma}_{h} \neq 0$, then by Remark~\ref{r: zeros of B sigma induced by Omega sigma} implies that $U \in \Omega_{\sigma}$. Therefore, if $g \in U$, then $\phi'(U,g) = \prod_{h \in U} e_{h}^{\sigma}[g]^{\sigma}$ and consequently $\phi'$ is surjective. Furthermore, $\phi'$ is injective. Indeed, observe that for all $U \in \Omega_{\sigma}$, such that $U = \{ 1, g, h_1, \ldots, h_n \}$ we have that
    \begin{align*}
        \phi \phi'(U,g)
        &= \phi(e_{h_1}^{\sigma}e_{h_2}^{\sigma} \ldots e_{h_n}^{\sigma}[g]^{\sigma}) \\
        (\text{by \eqref{eq: phi in basis of kparsG}})&= \chi(U)(U, g) \\
        &= (U, g).
    \end{align*}
    Since $\{ (U,g) : g \in G, \, U \in \Omega_{\sigma}, \text{ and } g \in U \}$ generates $\mathcal{A}_{\sigma}$ as $\kappa$-module, we conclude that $\phi \phi' = 1_{\mathcal{A}}$. Hence, $\phi'$ is injective. Consequently, $\phi'$ is a bijection, and $\phi$ is an isomorphism of algebras.
\end{proof}

Form Theorem~\ref{t: kparsG isomorphism} we conclude that 
\begin{equation}
    \xi = \{ 1, h_1, \ldots, h_n \} \in \Omega_\sigma \Leftrightarrow e_{h_1}^{\sigma}e_{h_2}^{\sigma} \ldots e_{h_n}^{\sigma} \neq 0.
    \label{eq: Omega sigma charaterization via kparsG}
\end{equation}

\begin{corollary} \label{c: type I prohibition charaterization}
    $\{ 1, g, h \} \notin \Omega_{\sigma}$ if, and only if, $\sigma(g^{-1},h)=0$.
    \label{c: basic zeros of B sigma depends only on sigma}
\end{corollary}

\begin{proof}
    Suppose that $\{ 1, g, h \} \notin \Omega_\sigma$. Then, by \eqref{eq: Omega sigma charaterization via kparsG} we have that $e_{g}^{\sigma} e_{h}^{\sigma} = 0$. Moreover,
    \[
        e_{g}^{\sigma} e_{h}^{\sigma} = 0 \Rightarrow [g^{-1}]^{\sigma} [h]^{\sigma} = 0 \Rightarrow \sigma(g^{-1}, h) = 0.
    \]
    In the other direction, if $\sigma(g^{-1}, h) = 0$, we obtain
    \[
      \{ 1, g, h \} = \{ g, g g^{-1}, g g^{-1} h \}  \in \mathcal{P}_{\sigma}.
    \]
\end{proof}

\section{Another structural characterization \texorpdfstring{$\kappa_{par}^{\sigma}G$}{TEXT}}\label{sec:Another}
Now we will give another characterization of the structure of $\kappa_{par}^{\sigma}G$ when $\Omega_{\sigma}$ is discrete.
We shall see in  Corollary~\ref{c: spectre is discrete if the admisibles of 1 is finite}, that in this case \( \Omega_{\sigma} \) is finite.
Consequently, by Lemma~\ref{l: Omega sigma contains subset elements}, we have that each \( \xi \in \Omega_{\sigma} \) is finite.

\textbf{In this section} we will assume that \(\Omega_{\sigma}\) is discrete. Let $\mathcal{R}^{\sigma}(G)$ be the $\kappa$-vector space with basis $\{ (g, A) : A \in \Omega_{\sigma} \text{ and } g^{-1} \in A \}$. Furthermore, $\mathcal{R}^{\sigma}(G)$ is a $\kappa$-algebra with the product
\begin{equation} \label{eq: product in Gamma G}
    (g, B)*(h, A) := \left\{\begin{matrix}
        \sigma(g, h)(gh, A) & \text{ if } B = hA, \\ 
        0 & \text{ otherwise.}
    \end{matrix}\right.
\end{equation}
We have to verify that $\mathcal{R}^{\sigma}(G)$ is associative with the product \eqref{eq: product in Gamma G}. Let $(g, A), \, (h, B)$ and $(t, C)$ be in $\mathcal{R}^{\sigma}(G)$, if $A \neq hB$ or $B \neq tC$, then $\big( (g, A) (h, B) \big) (t, C) = 0$. Indeed, observe that
\begin{align*}
    \big( (g, A) (h, B) \big) (t, C) = \left\{\begin{matrix}
        \sigma(g, h)(gh, B)(t, C) & \text{ if } A = hB, \\ 
        0 & \text{ if } A \neq hB.
    \end{matrix}\right.
\end{align*}
Hence, by the hypothesis we conclude that $\big( (g, A) (h, B) \big) (t, C) = 0$. Analogously, 
\begin{align*}
    (g, A) \big( (h, B) (t, C) \big) = \left\{\begin{matrix}
        \sigma(h, t)(g, A)(ht, C) & \text{ if } B = tC, \\ 
        0 & \text{ if } B \neq tC.
    \end{matrix}\right.
\end{align*}
Thus, if $B = tC$, then $A \neq hB$, which implies $A \neq htC$. Therefore, $(g, A) \big( (h, B) (t, C) \big) = 0$.

If $A = hB$ and $B = tC$, then $A = htC$, $\{ 1, g, gh, ght \} \subseteq gA$,
\[
    \big( (g, A) (h, B) \big) (t, C)
    = \sigma(g, h)(gh, B)(t, C)
    = \sigma(g, h)\sigma(gh, t)(ght, C)
\]
and
\[
    (g, A) \big( (h, B) (t, C) \big)
    = \sigma(h, t) (g, A)(ht, C) = \sigma(g, ht)\sigma(h, t) (ght, C).
\]
By definition, $A \in \Omega_{\sigma}$, and since $g^{-1} \in A$, we conclude that $gA \in \Omega_{\sigma}$. Hence, $\{1, g, gh, ght\} \subseteq gA$ is not a type II prohibition, and therefore $\sigma(g, h) \sigma(gh, t) = \sigma(g, ht) \sigma(h, t)$. Consequently, $\big( (g, A) (h, B) \big) (t, C) = (g, A) \big( (h, B) (t, C) \big)$. Henceforth, $\mathcal{R}^{\sigma}(G)$ is a well-defined associative algebra.
Observe that
\[
    \sum_{A \in \Omega_{\sigma}} (1, A)
\]
is the unit of $\mathcal{R}^{\sigma}(G)$. Indeed,
\[
    \left( \sum_{A \in \Omega_{\sigma}} (1, A) \right) * (g, B) = (1, gB) * (g, B) = \sigma(1, g) (g, B) = (g, B)
\]
and
\[
    (g, B) * \left( \sum_{A \in \Omega_{\sigma}} (1, A) \right)  = (g, B) (1, B)  = \sigma(g, 1)(g, B) = (g, B).
\]
Thus, $\mathcal{R}^{\sigma}(G)$ is a unital $\kappa$-algebra.
\begin{proposition}
    The map $\pi: G \to \mathcal{R}^{\sigma}(G)$ given by $\pi(g) := \sum_{A \ni g^{-1}}(g, A)$ is a partial projective $\sigma$-representation of $G$, where we tacitly assume $A \in \Omega_{\sigma}$.
\end{proposition}
\begin{proof}
    Suppose that $\pi(g^{-1}) \pi(g) \pi(h) \neq 0$, then
    \begin{align*}
        \pi(g^{-1}) \pi(g) \pi(h)
        &= \left( \sum_{A \ni g} (g^{-1}, A) \right) \left( \sum_{B \ni g^{-1}} (g, B) \right) \left( \sum_{C \ni h^{-1}} (h, C) \right) \\
        &= \sum_{C \ni h^{-1}, h^{-1}g^{-1}} (g^{-1}, ghC) (g, hC) (h, C) \\
        &= \sigma(g, h) \left(\sum_{C \ni h^{-1}, h^{-1}g^{-1}} (g^{-1}, ghC) (gh, C) \right)\\
        &= \sigma(g, h) \left(\sum_{A \ni g} (g^{-1}, A)\right)  \left(\sum_{C \ni h^{-1}g^{-1}} (gh, C) \right) \\
        &= \sigma(g, h) \pi(g^{-1}) \pi(gh).
    \end{align*}
    If $\pi(g^{-1}) \pi(g) \pi(h) = 0$, then there exists no set $C \in \Omega_{\sigma}$ such that $h^{-1}, h^{-1}g^{-1} \in C$. In particular, $\{ 1, h^{-1}, h^{-1} g^{-1} \} \notin \Omega_{\sigma}$, hence $\{ 1, h, g^{-1} \} \notin \Omega_\sigma$. Thus, by Corollary~\ref{c: type I prohibition charaterization} we conclude that $\sigma(g, h) = 0$. Analogously, one verifies that $\pi(g)\pi(h)\pi(h^{-1}) = \sigma(g, h) \pi(gh)\pi(h^{-1})$.
\end{proof}

By the universal property of $\kappa_{par}^{\sigma}G$ there exists a map of $\kappa$-algebras $\psi: \kappa_{par}^{\sigma}G \to \mathcal{R}^{\sigma}(G)$ such that
\begin{equation} \label{eq: psi formula isomorphism}
    \psi([g]^{\sigma}) = \sum_{A \ni g^{-1}} (g, A).
\end{equation}
For $\xi \in \Omega_{\sigma}$ we define
\begin{equation*}
    \Upsilon_\xi:= \prod_{s \in \xi} e_{s}^{\sigma} \prod_{t \notin \xi} (1 - e_{t}^{\sigma}).
\end{equation*}

Observe that since \(\Omega_{\sigma}\) is finite, then the sets
\(\mathfrak{A}_{\sigma}:=\{ g \in G : \{1,g\} \in \Omega_{\sigma} \} \) and $\xi$ are also finite. Moreover, if \(t \notin \mathfrak{A}_{\sigma}\) then \(e_{t}^{\sigma} = 0\), which implies that \(\Upsilon_{\xi}\) is well-defined for all \(\xi \in \Omega_{\sigma}\).
Hence,

\begin{align*}
    1_{\kappa_{par}^{\sigma}G}
    &= \prod_{g \in \mathfrak{A}_{\sigma}} 1
    = \prod_{g \in \mathfrak{A}_{\sigma}} (1 - e_{g}^{\sigma}  + e_{g}^{\sigma} ) \\
    &= \sum_{S \subseteq \mathfrak{A}_{\sigma}} \left( \prod_{g \in S} e_{g}^{\sigma} \right) \left( \prod_{g \notin S} 1 - e_{g}^{\sigma} \right) \\
    &=  \sum_{\xi   \in \Omega_{\sigma} } \left( \prod_{g \in \xi} e_{g}^{\sigma} \right) \left( \prod_{g \notin \xi} 1 - e_{g}^{\sigma} \right) \\
    &= \sum_{\xi \in \Omega_{\sigma} }  \Upsilon_{\xi}, 
\end{align*}
since if \(S \subseteq \mathfrak{A}_{\sigma}\) is not in \(\Omega_{\sigma}\) then \(\prod_{g \in S} e_{g}^{\sigma} = 0\).
Hence,
\begin{equation}
    \sum_{A \in \Omega_{\sigma}} \Upsilon_{A} = 1_{\mathcal{B}^{\sigma}} = 1_{\kappa_{par}^{\sigma}G}.
    \label{eq: sum of Upsilon is unit}
\end{equation}

\begin{lemma} \label{l: psi Upsilon computations}
    The morphism $\psi$ has the following properties:
    \begin{enumerate}[(i)]
        \item $\psi(e_{g_1}^{\sigma}e_{g_2}^{\sigma} \ldots e_{g_n}^{\sigma}) = \sum_{A \ni g_1, \ldots, g_n} (1, A)$,
        \item $\psi\big((1-e_{g_1}^{\sigma})(1-e_{g_2}^{\sigma}) \ldots (1- e_{g_n}^{\sigma}) \big) = \sum_{A \niton g_1, \ldots, g_n} (1, A)$,
        \item $\psi(\Upsilon_A) = (1, A)$ for all $A \in \Omega_{\sigma}$,
        \item $\psi([g]^{\sigma}e_{h_1}^{\sigma}e_{h_2}^{\sigma} \ldots e_{h_n}^{\sigma}) = \sum_{A \ni g^{-1}, h_1, \ldots, h_n}(g, A)$ for all $g, h_1, \ldots, h_n \in G$.
        \item $\psi([g]^{\sigma}\Upsilon_A) = (g, A)$ for all $g \in G$ and $A \in \Omega_{\sigma}$ such that $g^{-1} \in A$.
    \end{enumerate}
\end{lemma}
\begin{proof}
    Items $(i)$ and $(ii)$ are obtained by a direct computation. Item $(iii)$ is consequence of items $(i)$ and $(ii)$. 
    Item $(iv)$ is a consequence of item $(i)$ and equation \eqref{eq: psi formula isomorphism}.
    Finally, item $(v)$ is consequence of item $(iii)$ and equation \eqref{eq: psi formula isomorphism}.
\end{proof}

\begin{proposition}\label{prop:kpariso2}
    The map $\psi: \kappa_{par}^{\sigma}G \to \mathcal{R}^{\sigma}(G)$ is an isomorphism of algebras.
\end{proposition}
\begin{proof}
    Define the $\kappa$-linear map $\psi': \mathcal{R}^{\sigma}(G) \to \kappa_{par}^{\sigma}G$ such that $\psi'(g, A) = [g]^{\sigma} \Upsilon_{A}$. Then, $\psi$ and $\psi'$ are mutually inverses. Indeed, by item $(v)$ of Lemma~\ref{l: psi Upsilon computations} we have that $\psi \psi' = 1_{\mathcal{R}^{\sigma}(G)}$. For the converse, by item (iv) of Lemma~\ref{l: psi Upsilon computations} we see that
    \begin{align*}
        \psi' \psi([g]^{\sigma}e_{h_1}^{\sigma} \ldots e_{h_n}^{\sigma})
        &= \psi' \left( \sum_{A \ni g^{-1}, h_1, \ldots, h_n}(g, A) \right) \\
        &= \sum_{A \ni g^{-1}, h_1, \ldots, h_n} [g]^{\sigma}\Upsilon_A \\
        &= \sum_{A \ni g^{-1}, h_1, \ldots, h_n} [g]^{\sigma}e_{h_1}^{\sigma} \ldots e_{h_n}^{\sigma}\Upsilon_A \\
        &= [g]^{\sigma}e_{h_1}^{\sigma} \ldots e_{h_n}^{\sigma} \sum_{A \in \Omega_{\sigma}}\Upsilon_A \\
        &= [g]^{\sigma}e_{h_1}^{\sigma} \ldots e_{h_n}^{\sigma},
    \end{align*}
    where the latter equality holds by equation \eqref{eq: sum of Upsilon is unit}. Hence, $\psi$ is a bijection, which yields the desired conclusion.
\end{proof}

Let $\Gamma^{\sigma}(G) := \{ r(g, U) \in \mathcal{R}^{\sigma}(G), \, r \neq 0\}$. Note that $\Gamma^{\sigma}(G)$ is a groupoid with the product inherited from $\mathcal{R}^{\sigma}(G)$, establishing that two elements $r(g, B)$ and $l(h, A)$ in $\Gamma^{\sigma}(G)$ are composable if, and only if $B = hA$, and by Remark~\ref{r: sigma properties} this is equivalent to the fact that $(g,B) \star (h, A) \neq 0$.

Let $\{ \Gamma_i^{\sigma}(G) \}_{i \in I}$ be the set of connected components of $\Gamma^{\sigma}(G)$. We define $\mathcal{R}^{\sigma}_{i}(G)$ as $\kappa$-subspace of $\mathcal{R}^{\sigma}(G)$ generated by $\Gamma_i^{\sigma}(G)$. Analogously to $\mathcal{R}^{\sigma}(G)$, one have that $\mathcal{R}^{\sigma}_{i}(G)$ is a $\kappa$-algebra. Therefore,
\begin{equation}
    \mathcal{R}^{\sigma}(G) = \bigoplus_{i} \mathcal{R}_{i}^{\sigma}(G).
    \label{eq: decomposition of groupoids algebra}
\end{equation}

Let $\mathcal{G}$ be a connected component of $\Gamma^{\sigma}(G)$. Take an enumeration $\{ A_i \}_{i=1}^{n}$ of the set of objects of $\mathcal{G}$. Fix $A := A_1$. We define the stabilizer of $A$ by $H_A := \{ g \in G : gA = A \}$.
Observe that since $H_A$ is a subgroup of $G$ and $H_{A} \subseteq A \in \Omega_{\sigma}$ we can conclude that $\sigma_{H_A}:= \sigma|_{H_A \times H_A}: H_A \times H_A \to \kappa$ is a global $2$-cocycle of $H_A$.
Furthermore, observe that the isotropy group $G_{A}$ of $A$ is $\kappa^{*} \times H_A \times \{ A \}$ (recall that the product in $G_A$ is the product inherited from $\Gamma^{\sigma}(G)$). 
Let $\kappa G_{A}$ be the subalgebra of $\mathcal{R}^{\sigma}(G)$ generated by $G_{A}$. Then, a direct computation shows that the map 
\begin{equation}
   (r, g, A) \in \kappa G_A \mapsto rg \in \kappa^{\sigma_{H_A}} H_A
    \label{eq: twisted group algebra isomorphism}
\end{equation}
is an isomorphism of algebras, where $\kappa^{\sigma_{H_A}} H_A$ is the twisted group algebra of the group $H_A$ by the $2$-cocycle $\sigma_{H_A}$. In view of the isomorphism \eqref{eq: twisted group algebra isomorphism} we identify $G_A$ with the subset $\{ rg : r \in \kappa^{*}, \, g \in H_A \}$ of $\kappa^{\sigma_{H_A}} H_A$.

For $i \in \{ 2, \ldots, n\}$ fix an arrow (morphism) $A \overset{\gamma_i}{\to} A_i$. Then, for any arrow $A_i \overset{\gamma}{\to} A_j$ we have that
\begin{equation} \label{eq: gamma decomposition}
    \gamma = \gamma_j( \gamma_j^{-1} \gamma  \gamma_i)\gamma_i^{-1}.
\end{equation}
Thus, by equation~\eqref{eq: gamma decomposition} for all $\gamma \in \mathcal{G}$ there exists a unique $z_{\gamma}:= \gamma_j^{-1} \gamma  \gamma_i \in G_A$ such that $\gamma = \gamma_j z_{\gamma}\gamma_i^{-1} $.
Denote by $\mathcal{R}_{\mathcal{G}}$ the subalgebra of $\mathcal{R}^{\sigma}(G)$ generated by $\mathcal{G}$ and define the $\kappa$-linear map
\begin{align*}
    \tau_{\mathcal{G}}: \mathcal{R}_\mathcal{G} &\to M_n(\kappa^{\sigma_{H_A}} H_A) \\
                            \gamma &\mapsto z_\gamma E_{j, i},
\end{align*}
where $E_{i,j}$ represents a matrix with zero entries everywhere except at the $(i,j)$-entry, where it equals $1$, i.e., $\{ E_{i, j} \}_{i,j = 1}^{n}$ is the canonical basis of $M_n(\kappa^{\sigma_{H_A}} H_A)$ as a left free $\kappa^{\sigma_{H_A}} H_A$-module. It is readily seen that $\tau_{\mathcal{G}}$ is $\kappa$-linear isomorphism. Moreover, let $A_i \overset{\gamma}{\to} A_j$ and $A_s \overset{\gamma'}{\to} A_t$ be in $\mathcal{G}$. Then,
\[
    \tau_{\mathcal{G}}(\gamma) \tau_{\mathcal{G}}(\gamma') = z_\gamma z_{\gamma'} \delta_{i,t} E_{j,s} = \tau_{\mathcal{G}}(\gamma \gamma'),
\]
where $\delta_{i,t}$ is the Kronecker delta. Whence, we conclude that $\tau_{\mathcal{G}}$ is an isomorphism of algebras. For any connected component $\Gamma_{i}^{\sigma}(G)$ of $\Gamma^{\sigma}(G)$ write $n_i = | \operatorname{Ob}(\Gamma_{i}^{\sigma}(G)) |$, let $H_i$ be the stabilizer of an arbitrary object of $\Gamma_{i}^{\sigma}(G)$ and $\sigma_i = \sigma|_{H_i \times H_i}$. Then, by equation~\eqref{eq: decomposition of groupoids algebra} and the aforementioned isomorphism we obtain

\begin{theorem} \label{t: groupoid algebra isomorphism}
 Assume that $\Omega_{\sigma}$ is discrete. Then   $$\kappa_{par}^{\sigma}G \cong \bigoplus_{i\in I} M_{n_i}(\kappa^{\sigma_i} H_i).$$
\end{theorem}

\begin{corollary}
    Let \(G\) a finite group and \(\sigma\) any factor set of \(G\).
    Then,
    $\kappa_{par}^{\sigma}G \cong \bigoplus_{i\in I} M_{n_i}(\kappa^{\sigma_i} H_i)$.
\end{corollary}

\begin{corollary} \label{c: simple twisted partial group algebra}
    $\kappa_{par}^{\sigma}G$ is simple, if and only if, $\sigma(g,1)=0$ for all $g \in G \setminus \{ 1 \}$.
\end{corollary}
\begin{proof}
    It is immediate that if $\sigma(g,1)=0$ for all $g \in G \setminus \{ 1 \}$ then $\kappa_{par}^{\sigma}G = \kappa$. On the other hand, notice that $\kappa$ is always a direct summand of $\kappa_{par}^{\sigma}G$ associated to the trivial connected component of $(1, \{ 1 \})$. Now, suppose that there exists $g \in G \setminus \{ 1 \}$ such that $\sigma(g,1) = 1$. Then, $\{ r(g, \{ 1, g^{-1} \}), \, l(g^{-1}, \{ 1, g \}) : r,l \in \kappa \setminus \{ 0 \} \}$ is a connected component of $\Gamma^{\sigma}(G)$. Hence, by Theorem~\ref{t: groupoid algebra isomorphism} we obtain another non-zero ideal, different from the ideal associated to $(1, \{ 1 \})$.
\end{proof}

\begin{remark}
    It is important to note that the condition $\sigma(g,1)=0$ for all $g \in G \setminus \{ 1 \}$ in Corollary~\ref{c: simple twisted partial group algebra} implies that $\sigma(g,h)=0$ for all $g,h \in G$ such that $g \neq 1$ or $g \neq 1$. Indeed, if $\sigma(g,1) = 0$ then $\{ 1, g, h^{-1} \} \notin \Omega_{\sigma}$, thus by Corollary~\ref{c: basic zeros of B sigma depends only on sigma} we conclude that $\sigma(g,h)=0$.
\end{remark}

\begin{remark}
    Note that the construction of the isomorphism $\tau_{\mathcal{G}}$ may depend on the choice of the object $A$. Let $A, B \in \operatorname{Ob}(\mathcal{G})$ such that $A \neq B$. Then, there exists $g \in G$ such that $B = gA$. Thus, $H_B = g H_A g^{-1}$ and the map
    \[
       h \in \kappa^{\sigma_{H_A}} H_A \mapsto g \cdot h \cdot g^{-1} \in \kappa^{\sigma_{H_B}} H_B
    \]
    is an isomorphism of algebras, where $\cdot$ stands for the product in $\kappa^{\sigma_{H_B}}H_B$. Therefore, the direct summands in Theorem~\ref{t: groupoid algebra isomorphism} do not depend (up to isomorphism) on the choice of the fixed object in the connected components of $\Gamma^{\sigma}(G)$.
\end{remark}

\begin{remark}
    Let $H$ be a subgroup of $G$ such that $\sigma|_{H \times H}$ is a global $2$-cocycle. Then, $H \in \Omega_{\sigma}$ and the connected component containing the object $H$ consists of the single object $H$ with isotropy group $\kappa^{*} \times H$.
\end{remark}

\section[Invariance of type II prohibitions]{\(S_4\)-Invariance of type II \(\sigma\)-prohibitions}\label{sec:S4-Invariance}

{\it In this section   $\kappa$ will be an  algebraically closed field.} Then by \cite[Proposition 4.5]{Dokuchaev-Jerez:twisted:2024} we can assume, without loss of generality, that $\sigma(g, g^{-1}) \in \{ 0, 1 \}$ for all $g \in G$, and consequently 
\begin{equation} \label{eq: sigma involution formula}
    \sigma(x, y)^{-1} = \sigma(y^{-1}, x^{-1})
\end{equation}
if $\sigma(x,y) \neq 0$, and
\begin{equation} \label{eq: C3 Orbit invariance}
    \sigma(x, y) = \sigma(y, y^{-1} x^{-1}) = \sigma(y^{-1} x^{-1}, x).
\end{equation}
Define the map $\partial(\sigma): G^{3} \to \kappa$ such that
\[
    \partial(\sigma)(x,y,z):= \sigma(x,y) \sigma(xy, z) \sigma(z^{-1} y^{-1}, x^{-1}) \sigma(z^{-1}, y^{-1}).
\]
Note that if $\partial(\sigma)(x,y,z) \neq 0$ then
\[
    \sigma(x,y)\sigma(xy, z) = \sigma(x, yz) \sigma(y,z) \Leftrightarrow \partial(\sigma)(x,y,z)=1.
\]

\begin{lemma}
    Let \(\sigma \in pm(G)\) as above and \(x,y,z \in G\) such that \(\partial(\sigma)(x,y,z) \neq 0 \).
    Then, $\partial(\sigma)(x, y, z) = 1$ if $x = 1$ or $y = 1$ or $z = 1$.
\end{lemma}
\begin{proof}
    Observe that
    \[
        \partial(\sigma)(1, y, z) = \sigma(1,y) \sigma(y, z) \sigma(z^{-1} y^{-1}, 1) \sigma(z^{-1}, y^{-1}) = \sigma(y, z)\sigma(z^{-1}, y^{-1}) = 1.
    \]
    The other cases are analogous.
\end{proof}

We define the following equivalence relation in $G^{3}$, we say that $(x,y,z) \sim_{S_4} (a,b,c)$ if, and only if, there exist $g \in G$ such that $\{ 1, x, xy, xyz \} = \{ g, ga, gab, gabc \}$. 
Indeed, this relation is obviously reflexive and symmetric. For transitivity, observe that if $(x,y,z) \sim_{S_4} (a,b,c)$ and $(a,b,c) \sim_{S_4} (u,v,w)$, then there exist $g,h \in G$ such that
\[
   \{ 1, x, xy, xyz \} = \{ g, ga, gab, gabc \} \text{ and } \{ 1, a, ab, abc \} =  \{ h, hu, huv, huvw \}, 
\]
but this implies that $\{ 1, x, xy, xyz \} = \{ gh, ghu, ghuv, ghuvw \}$ and therefore $(x,y,z) \sim_{S_4} (u,v,w)$.

Fix an arbitrary element $(x,y,z) \in G^{3}$.
We want to determine the class of $(x,y,z)$ under $\sim_{S_4}$.
If $(a,b,c) \sim_{S_4} (x,y,z)$, then there exist $g \in G$ such that $(g, ga, gab, gabc) = (v_0, v_1, v_2, v_3)$, where $(v_0, v_1, v_2, v_3)$ is some permutation of the tuple $(1, x, xy, xyz)$.
Thus, we conclude that
\begin{equation} \label{eq: solveFor} 
    g = v_0, \, a = v_0^{-1}v_1, \, b = v_1^{-1}v_2, \, c = v_2^{-1}v_3.
\end{equation}
Then, we have to analyze a total of $24$ cases each one corresponding to one permutation of $(1, x, xy, xyz)$.
We have to compute \eqref{eq: solveFor} for each permutation of $(1, x, xy, xyz)$.
The following table shows the computations in the following form:
\[
    \Big(\text{permutation of } (1, x, xy, xyz) \Big) \to \Big(\text{ corresponding solution }(g, a, b, c) \Big)
\]
\begin{minipage}{0.4\textwidth}
    \begin{align*}
        (1, x, xy, xyz)  & \to (1, x, y, z) \\
        (1, x, xyz, xy)  & \to (1, x, yz, z^{-1}) \\
        (1, xy, x, xyz)  & \to (1, xy, y^{-1}, yz) \\
        (1, xy, xyz, x)  & \to (1, xy, z, z^{-1}y^{-1}) \\
        (1, xyz, x, xy)  & \to (1, xyz, z^{-1}y^{-1}, y) \\
        (1, xyz, xy, x)  & \to (1, xyz, z^{-1}, y^{-1}) \\
        (x, 1, xy, xyz)  & \to (x, x^{-1}, xy, z) \\
        (x, 1, xyz, xy)  & \to (x, x^{-1}, xyz, z^{-1}) \\
        (x, xy, 1, xyz)  & \to (x, y, y^{-1}x^{-1}, xyz) \\
        (x, xy, xyz, 1)  & \to (x, y, z, z^{-1}y^{-1}x^{-1}) \\
        (x, xyz, 1, xy)  & \to (x, yz, z^{-1}y^{-1}x^{-1}, xy) \\
        (x, xyz, xy, 1)  & \to (x, yz, z^{-1}, y^{-1}x^{-1})
    \end{align*}
\end{minipage}
\begin{minipage}{0.5\textwidth}
    \begin{align*}            
        (xy, 1, x, xyz)  & \to (xy, y^{-1}x^{-1}, x, yz) \\
        (xy, 1, xyz, x)  & \to (xy, y^{-1}x^{-1}, xyz, z^{-1}y^{-1}) \\
        (xy, x, 1, xyz)  & \to (xy, y^{-1}, x^{-1}, xyz) \\
        (xy, x, xyz, 1)  & \to (xy, y^{-1}, yz, z^{-1}y^{-1}x^{-1}) \\
        (xy, xyz, 1, x)  & \to (xy, z, z^{-1}y^{-1}x^{-1}, x) \\
        (xy, xyz, x, 1)  & \to (xy, z, z^{-1}y^{-1}, x^{-1}) \\
        (xyz, 1, x, xy)  & \to (xyz, z^{-1}y^{-1}x^{-1}, x, y) \\
        (xyz, 1, xy, x)  & \to (xyz, z^{-1}y^{-1}x^{-1}, xy, y^{-1}) \\
        (xyz, x, 1, xy)  & \to (xyz, z^{-1}y^{-1}, x^{-1}, xy) \\
        (xyz, x, xy, 1)  & \to (xyz, z^{-1}y^{-1}, y, y^{-1}x^{-1}) \\
        (xyz, xy, 1, x)  & \to (xyz, z^{-1}, y^{-1}x^{-1}, x) \\
        (xyz, xy, x, 1)  & \to (xyz, z^{-1}, y^{-1}, x^{-1})
    \end{align*}
\end{minipage}

Therefore, the equivalence class of an arbitrary element $(x, y, z) \in G^{3}$ is
\begin{align*}
    \overline{(x, y, z)} = \Big\{ 
    & (x, y, z),
    (x, yz, z^{-1}), 
    (xy, y^{-1}, yz), \\ 
    &(xy, z, z^{-1}y^{-1}), 
    (xyz, z^{-1}y^{-1}, y), 
    (xyz, z^{-1}, y^{-1}), \\ 
    &(x^{-1}, xy, z),
    (x^{-1}, xyz, z^{-1}), 
    (y, y^{-1}x^{-1}, xyz), \\
    &(y, z, z^{-1}y^{-1}x^{-1}), 
    (yz, z^{-1}y^{-1}x^{-1}, xy), 
    (yz, z^{-1}, y^{-1}x^{-1}) \\
    &(y^{-1}x^{-1}, x, yz), 
    (y^{-1}x^{-1}, xyz, z^{-1}y^{-1}),
    (y^{-1}, x^{-1}, xyz), \\
    &(y^{-1}, yz, z^{-1}y^{-1}x^{-1}), 
    (z, z^{-1}y^{-1}x^{-1}, x), 
    (z, z^{-1}y^{-1}, x^{-1}), \\
    &(z^{-1}y^{-1}x^{-1}, x, y), 
    (z^{-1}y^{-1}x^{-1}, xy, y^{-1}), 
    (z^{-1}y^{-1}, x^{-1}, xy), \\
    &(z^{-1}y^{-1}, y, y^{-1}x^{-1}), 
    (z^{-1}, y^{-1}x^{-1}, x), 
    (z^{-1}, y^{-1}, x^{-1})                
    \Big\}.
\end{align*}
Furthermore, we can define a map $\triangleright : G^{3} \times S_{4} \to G^{3}$ such that
\[
   \gamma \triangleright (x,y,z) := (a,b,c)
\]
where $(a,b,c)$ is given by \eqref{eq: solveFor} for the permutation $\gamma$. That is, $(v_0, v_1, v_2, v_3) = (u_{\gamma^{-1}(0)}, u_{\gamma^{-1}(1)}, u_{\gamma^{-1}(2)}, u_{\gamma^{-1}(3)})$, where $(u_0, u_1, u_2, u_3) = (1, x, xy, xyz)$.
Therefore, we obtain the following table for the function~$\triangleright$:

\begin{minipage}{0.5\textwidth}
    \begin{align*}
        (0)          \triangleright X  & =  (x, y, z) \\
        (2, 3)       \triangleright X  & =  (x, yz, z^{-1}) \\
        (1, 2)       \triangleright X  & =  (xy, y^{-1}, yz) \\
        (1, 3, 2)    \triangleright X  & =  (xy, z, z^{-1}y^{-1}) \\
        (1, 2, 3)    \triangleright X  & =  (xyz, z^{-1}y^{-1}, y) \\
        (1, 3)       \triangleright X  & =  (xyz, z^{-1}, y^{-1}) \\
        (0, 1)       \triangleright X  & =  (x^{-1}, xy, z) \\
        (0, 1)(2, 3) \triangleright X  & =  (x^{-1}, xyz, z^{-1}) \\
        (0, 2, 1)    \triangleright X  & =  (y, y^{-1}x^{-1}, xyz) \\
        (0, 3, 2, 1) \triangleright X  & =  (y, z, z^{-1}y^{-1}x^{-1}) \\
        (0, 2, 3, 1) \triangleright X  & =  (yz, z^{-1}y^{-1}x^{-1}, xy) \\
        (0, 3, 1)    \triangleright X  & =  (yz, z^{-1}, y^{-1}x^{-1})
    \end{align*}
\end{minipage}
\begin{minipage}{0.5\textwidth}
    \begin{align*}
        (0, 1, 2)    \triangleright X  & =  (y^{-1}x^{-1}, x, yz) \\
        (0, 1, 3, 2) \triangleright X  & =  (y^{-1}x^{-1}, xyz, z^{-1}y^{-1}) \\
        (0, 2)       \triangleright X  & =  (y^{-1}, x^{-1}, xyz) \\
        (0, 3, 2)    \triangleright X  & =  (y^{-1}, yz, z^{-1}y^{-1}x^{-1}) \\
        (0, 2)(1, 3) \triangleright X  & =  (z, z^{-1}y^{-1}x^{-1}, x) \\
        (0, 3, 1, 2) \triangleright X  & =  (z, z^{-1}y^{-1}, x^{-1}) \\
        (0, 1, 2, 3) \triangleright X  & =  (z^{-1}y^{-1}x^{-1}, x, y) \\
        (0, 1, 3)    \triangleright X  & =  (z^{-1}y^{-1}x^{-1}, xy, y^{-1}) \\
        (0, 2, 3)    \triangleright X  & =  (z^{-1}y^{-1}, x^{-1}, xy) \\
        (0, 3)       \triangleright X  & =  (z^{-1}y^{-1}, y, y^{-1}x^{-1}) \\
        (0, 2, 1, 3) \triangleright X  & =  (z^{-1}, y^{-1}x^{-1}, x) \\
        (0, 3)(1, 2) \triangleright X  & =  (z^{-1}, y^{-1}, x^{-1}).                
    \end{align*}
\end{minipage}
To verify that $\triangleright$ is a group action of $S_4$ on $G^{3}$ it is enough to check it on the basic products of some generator set, for example take the generators $(0, 1, 2, 3)$ and $(1, 2)$. Hence, observe that
\begin{align*}
    &(0,1,2,3) \triangleright (0,1,2,3) \triangleright X= (0,1,2,3) \triangleright (z^{-1}y^{-1}x^{-1}, x, y) =  (z, z^{-1}y^{-1}x^{-1}, x)\\
    &(0, 1, 2, 3)(0,1,2,3) \triangleright X = (0,2)(1,3) \triangleright X = (z, z^{-1}y^{-1}x^{-1}, x).
\end{align*}

\begin{align*}
    &(1,2) \triangleright (0,1,2,3) \triangleright X = (1,2) \triangleright (z^{-1}y^{-1}x^{-1}, x, y) =  (z^{-1}y^{-1}, x^{-1}, xy)\\
    &(1, 2)(0,1,2,3) \triangleright X = (0,2,3) \triangleright X = (z^{-1}y^{-1}, x^{-1}, xy).
\end{align*}

\begin{align*}
    &(1,2) \triangleright (1,2) \triangleright X = (1,2) \triangleright (xy, y^{-1}, yz) =  (x, y, z)\\
    &(1, 2)(1,2) \triangleright X = (0) \triangleright X = X.
\end{align*}

\begin{align*}
    &(0,1,2,3) \triangleright (1,2) \triangleright X= (0,1,2,3) \triangleright (xy, y^{-1}, yz) =  (z^{-1}y^{-1}x^{-1}, xy, y^{-1})\\
    &(0, 1, 2, 3)(0,1,2,3) \triangleright X = (0,1,3) \triangleright X = (z^{-1}y^{-1}x^{-1}, xy, y^{-1}).
\end{align*}
Thus, we have an action of $S_4$ on $G^{3}$. Moreover, the relation $\sim_{S_4}$ corresponds to the relation determined by the action $\triangleright$.

\begin{proposition}\label{prop:S4-action on delta}
   Assume that the field  $\kappa $ is algebraically closed and let $A_4$ be the alternating subgroup of $S_4$. Then,
    \begin{equation}
       \partial(\sigma)(\gamma \triangleright (x,y,z)) = \partial(\sigma)(x,y,z) \text{ for all } \gamma \in A_4.
        \label{eq: A4 invariance}
    \end{equation}
    Furthermore, $\partial(\sigma)(x,y,z) = 0$ if and only if $\partial(\sigma)((0,2) \triangleright (x,y,z)) = 0$. Moreover, if $\partial(\sigma)(x,y,z) \neq 0$ then
    \[
        \partial(\sigma)((0,2) \triangleright (x,y,z)) = \partial(\sigma)(x,y,z)^{-1}.
    \]
\end{proposition}
\begin{proof}
    We construct the following table, placing the factors of $\partial(\sigma)(x,y,z)$ in the first row.
    In the columns, we list the invariant elements corresponding to each factor of $\partial(\sigma)(x,y,z)$, as determined by \eqref{eq: C3 Orbit invariance}.
    \begin{equation} \label{eq: s3-table}
        \begin{matrix}
            \sigma(x, y)              & \sigma(xy, z)                 &  \sigma(z^{-1}y^{-1}, x^{-1})& \sigma(z^{-1}, y^{-1}) \\ 
            \sigma(y, y^{-1} x^{-1})  & \sigma(z, z^{-1}y^{-1}x^{-1}) &  \sigma(xyz, z^{-1}y^{-1})   & \sigma(yz, z^{-1})\\ 
            \sigma(y^{-1} x^{-1}, x)  & \sigma(z^{-1}y^{-1}x^{-1}, xy)&  \sigma(x^{-1}, xyz)         & \sigma(y^{-1}, yz)
        \end{matrix}
    \end{equation}
    Note that we can express $\partial(\sigma)(x, y, z)$ as any product involving one element from each column. Explicitly, $\partial(\sigma)(x, y, z) = A \cdot B \cdot C \cdot D$, where $A$ is in the first column, $B$ is in the second column, $C$ is in the third column, and $D$ is in the fourth column. Thus, for the first part of the proposition it is enough to check the equality \eqref{eq: A4 invariance} in some generators of $A_4$, for example, take $(0,1,2)$ and $(0,3)(1,2)$. Thus,
    \begin{align*}
        \partial(\sigma)((0,1,2) \triangleright (x, y, z)) 
        &= \partial(\sigma)(y^{-1}x^{-1}, x, yz)  \\
        &= \sigma(y^{-1}x^{-1}, x) \sigma(y^{-1}, yz) \sigma(z^{-1}y^{-1}x^{-1}, xy) \sigma(z^{-1}y^{-1}, x^{-1})
    \end{align*}
    \begin{align*}
        \partial(\sigma)((0, 3)(1, 2) \triangleright (x, y, z)) 
        &= \partial(\sigma)(z^{-1}, y^{-1}, x^{-1})  \\
        &= \sigma(z^{-1}, y^{-1}) \sigma(z^{-1}y^{-1}, x^{-1}) \sigma(xy, z)\sigma(x, y),
    \end{align*}
    we can easily verify that both of the above cases can be obtained from table \eqref{eq: s3-table}.
    For the second part note that
    \[
        \partial(\sigma)((0, 2) \triangleright (x, y, z)) = \partial(\sigma)(xy, y^{-1}, yz)   =    \sigma(xy, y^{-1}) \sigma(x, yz) \sigma(z^{-1}, y^{-1}x^{-1}) \sigma(z^{-1}y^{-1}, y),
    \]
    observe that each factor of $\partial(\sigma)(xy, y^{-1}, yz)$ is equal to $\sigma(b^{-1}, a^{-1})$, where $\sigma(a,b)$ is in table \eqref{eq: s3-table}. Moreover, the corresponding elements are in different columns.
    Therefore, (v) of Remark~\ref{r: sigma properties} and \eqref{eq: sigma involution formula} yield the desired conclusion. 
\end{proof}

\begin{corollary} \label{c: Q s4-invariance}
    Assume that the field  $\kappa $ is algebraically closed and let $(x, y, z) \in G^{3}$. Then $\partial(\sigma)(x, y, z) = 0$ if and only if $\partial(\sigma)(\gamma \triangleright (x, y, z)) = 0$ for all $\gamma \in S_4$. Furthermore, $\partial(\sigma)(x,y,z) = 1$ if and only if $\partial(\sigma)(\gamma \triangleright (x, y, z)) = 1$ for all $\gamma \in S_4$.
\end{corollary}

\begin{proposition}\label{prop:refinement}
     Assume that the field  $\kappa $ is algebraically closed and let  $\sigma$ be a factor set with total domain. Then,
    \[
        [x]^{\sigma} [y]^{\sigma} [z]^{\sigma} = 0 \Leftrightarrow \sigma(x,y)\sigma(xy, z) \neq \sigma(x, yz) \sigma(y,z).
    \]
\end{proposition}
\begin{proof}
    We already know that if $\sigma(x,y)\sigma(xy, z) \neq \sigma(x, yz) \sigma(y,z)$ then $[x]^{\sigma} [y]^{\sigma} [z]^{\sigma} = 0$. Conversely, suppose that $[x]^{\sigma} [y]^{\sigma} [z]^{\sigma} = 0$, therefore $e_x^{\sigma} e_{xy}^{\sigma}e_{xyz}^{\sigma} = 0$. Then, by Theorem~\ref{t: kparsG isomorphism} we have that $\{ 1, x, xy, xyz \} \in \mathcal{P}^{II}_{\sigma}$.
    Hence, there exist $g,a,b,c \in G$ such that $\{ 1, x, xy, xyz \} = \{ g, ga, gab, gabc \}$ and $\partial(\sigma)(a,b,c) \neq 1$, but by Corollary~\ref{c: Q s4-invariance} we conclude that $\partial(\sigma)(x,y,z) \neq 1$, thus $\sigma(x,y)\sigma(xy, z) \neq \sigma(x, yz) \sigma(y,z)$.
\end{proof}

\section{Topological freeness of the spectral partial action}\label{sec:TopolFree}

In all what follows  \(\kappa\) will be an arbitrary field.

\begin{definition}\label{def:TopolFree}
    A topological partial action $\theta = (\{ X_{g} \}, \{ \theta_{g} \})$ of a group $G$ on a topological space $X$ is called \textbf{topologically free} if for all $g \neq 1_{G}$ the set $\{ x \in X_{g^{-1}} : \theta_{g}(x) = x \}$ has empty interior.
\end{definition}

Let \( X \) be a compact, locally disconnected Hausdorff topological space. Recall that \( \mathscr{L}(X) \) is the commutative algebra of all continuous functions from \( X \) to \( \kappa \), where \( \kappa \) has the discrete topology. Then, as previously mentioned, \( \mathscr{L}(X) \) is the algebra of all locally constant functions. From \cite{Dokuchaev-Exel:IdealStructure:2017}, we know that any topological partial group action on \( \theta= (G, X, \{ X_g \}, \{ \theta_{g} \}) \) such that $X_g$ is closed for all $g \in G$, gives rise to a unital partial group action $\alpha$ on $A = \mathscr{L}(X)$ as follows:
\begin{itemize}
    \item $\alpha:=(G, A, \{ A_{g} \}, \{ \alpha_g \})$,
    \item $A_{g}:= \{ f: X_g \to \kappa \text{ continuous}: f|_{X \setminus X_g} = 0 \}  \cong \{ f: X_g \to \kappa : f \text{ is continuous}\}$,
    \item each ideal $A_{g}$ is generated by the idempotent $1_g$ which is the characteristic function of $X_g$.
    \item $\alpha_{g}(f) := f \circ \theta_{g^{-1}}$,
\end{itemize}

Furthermore, given a group \(G\) and a unital algebra \(A\) generated by idempotents, in view of the above construction, and by Remark~\ref{r: spectral partial group action} and Theorem~\ref{t: JacobsonKeimel}, we obtain a bijection between the unital partial actions of \(G\) on \(A\) and the topological partial actions of \(G\) on \(\operatorname{Spec} A\) with clopen domains.

The following proposition is analogous to \cite[Proposition 4.7]{Goncalves-Oinert-Royer:2014:JA}, and the proof is completely similar; however, we will include it for the sake of completeness.

\begin{proposition}\label{prop:MaxCommut}
    Suppose that $\theta= (\{ X_{g} \}_{g \in G}, \{ \theta_{g} \}_{g})$ is a topologically free partial group action such that each \(X_{g}\) is clopen.
    Let \(\alpha\) be the partial action of \(G\) on \(\mathscr{L}(X)\) induced by \(\theta\).
    Let, furthermore, \(\sigma: G \times G \to \kappa\) be {red} a  twist of \(\alpha\).
    Then, $\mathscr{L}(X)\delta_{1_{G}}$ is a maximal commutative subalgebra of $\mathscr{L}(X) \rtimes_{(\alpha, \sigma)} G$.
\end{proposition}

\begin{proof}
    Let $A := \mathscr{L}(X)$ and $\alpha=\{ \alpha_{g}: A_{g^{-1}} \to A_{g} \}_{g \in G}$ be the partial action induced by $\theta$.
    Suppose that $A \delta_{1}$ is not maximal commutative.
    Then, there exists a non-zero function $f_{g} \in A_{g}$, with $g \neq 1_{G}$, such that $(f_{g}\delta_{g}) (f \delta_{1}) = (f \delta_{1})(f_{g}\delta_{g})$, which is equivalent to  $f_{g}\alpha_{g}(1_{g^{-1}}f) = f f_{g}$, for all $f \in A$, as $\sigma(g,1)=\sigma(1,g)=1$.
    This means that
    \begin{equation} \label{eq: maximal commutative}
        f_{g}(z) f(\theta_{g^{-1}}(z)) = f(z) f_{g}(z)
    \end{equation}
    for all $f \in A$ and $z \in X_{g}$.
    Since $f_{g}$ is non-zero and continuous, there exists $x \in X_{g}$ and a neighborhood $U_{x} \subseteq X_{g}$ of $x$ such that $f_{g}(y) \neq 0$ for all $ y \in U_{x}$.
    Thanks to the fact that the partial action is topologically free, there exist $y \in U_{x}$ such that $\theta_{g^{-1}}(y) \neq y$. 
    Since \( X \) is compact, totally disconnected and Hausdorff, there exists a compact open neighborhood \( V_{y} \) of \( y \) such that \( \theta_{g^{-1}}(y) \notin V_{y} \). However, for \( f = 1_{V_{y}} \), equation \eqref{eq: maximal commutative} implies that \( f_{g}(y) = 0 \), which leads to a contradiction.
\end{proof}

The proof of the following fact is a simple adaptation of the first part of the proof of \cite[Theorem 2.1]{Goncalves-Oinert-Royer:2014:JA}.

\begin{proposition} \label{p: top free importance}
    Let $A$ be a commutative associative algebra, $G$ a group and $\alpha$ a unital partial action of $G$ on $A$. Let, furthermore, \(\sigma: G \times G \to \kappa\) be  a  twist of \(\alpha\). If $A\delta_{1}$ is a maximal commutative subalgebra of $A \rtimes_{(\alpha, \sigma)} G$, then $\mathcal{I} \cap A \delta_{1} \neq 0$ for each non-zero ideal $\mathcal{I}$ of $A \rtimes_{(\alpha, \sigma)} G$.
\end{proposition}
\begin{proof}
    For each \(x = \sum a_{g}\delta_{g} \in A \rtimes_{(\alpha, \sigma)} G\) we define
    \[
       \operatorname{Supp}(x):= \{ g \in G : a_{g} \neq 0  \}.
    \]
    Let \(\mathcal{I}\) be a non-zero ideal of \(A \rtimes_{(\alpha, \sigma)} G\), and let \(x \in \mathcal{I}\) be a non-zero element, such that \(|\operatorname{Supp}(x)|\) is minimal among all non-zero elements of \(\mathcal{I}\).
    Fix \(s \in \operatorname{Supp}(x)\), and let
    \[
        y:= x \cdot 1_{s^{-1}}\delta_{s^{-1}} = \sigma(s,s^{-1})a_{s} \delta_{1} + \sum_{g \neq s} \sigma(g,s^{-1})a_{g}1_{gs^{-1}}\delta_{gs^{-1}} \in \mathcal{I}.
    \]
    Note that \(y \neq 0\).  Since \(|\operatorname{Supp}(x)|\) is minimal, then \(|\operatorname{Supp}(y)| = |\operatorname{Supp}(x)|\).
    Let \(a \in A\), then,  due to the commutativity of $A,$ the element \(z:= a \delta_{1} y - y a \delta_{1} \in \mathcal{I}\) is such that \(|\operatorname{Supp}(z)| < |\operatorname{Supp}(x)|\).
    Thus, the minimality of \(|\operatorname{Supp}(x)|\) implies that \(a \delta_{1} y = y a \delta_{1}\) for all \(a \in A\).
    Hence, the maximal commutativity of \(A \delta_{1}\) implies that \(y \in A \delta_{1}\), and therefore \(\mathcal{I} \cap A \delta_{1}  \neq 0\).
\end{proof}

Since $\kappa_{par}^{\sigma}G \cong B^{\sigma} \rtimes_{\sigma}G$ the next fact immediately follows from Proposition~\ref{prop:MaxCommut} and Proposition~\ref{p: top free importance}

\begin{corollary}\label{cor:TopolFreeImportance2}
    If the spectral partial group action  $\hat{\theta}$ is topologically free, then $\mathcal{I} \cap B^{\sigma} \neq 0$ for each non-zero ideal $\mathcal{I}$ of $\kappa_{par}^{\sigma}G$.
\end{corollary}

 Corollary~\ref{cor:TopolFreeImportance2}  highlights the importance of determining when the topological partial group action of \(G\) on \(B^{\sigma}\) is topologically free, as it gives relevant information about the ideals of the twisted partial group algebra \(\kappa_{par}^{\sigma}G\).

Let $\sigma \in pm(G)$, and $\hat{\theta}$ the spectral partial action of $G$ on $\Omega_{\sigma} \cong \operatorname{Spec} B^{\sigma}$. For $g \in G$ we set
\[
    \operatorname{Fix}_g^{\sigma} := \{ \xi \in \hat{D}_{g^{-1}} : \hat{\theta}_g(\xi) = \xi \}.
\]

For the case $\sigma = 1$ we just write $\operatorname{Fix}_{g}$ to denote $\operatorname{Fix}_{g}^{1}$.

\begin{remark}
    It is clear that $\xi \in \operatorname{Fix}_g^{\sigma}$ if, and only if, there exists $\{ h_i \}_{i \in I} \subseteq G \setminus \{ 1 \}$ such that $\xi = \Big(\bigsqcup_{i \in I} \{ g^{n}h_i \}_{n \in \mathbb{Z}}\Big) \sqcup \{ g^{n} \}_{n \in \mathbb{Z}} \in \Omega_{\sigma}$. Thus, if $|g| = \infty$ then $\xi$ is an infinite set, since $ \{ g^{n} \}_{n \in \mathbb{Z}} \subseteq \xi$.
    \label{r: invariant g point charaterization}
\end{remark}

 Recall that by Remark~\ref{r: Omega sigma is discrete when G is finite} the topology of $\Omega_{\sigma}$ is discrete when $G$ is finite. Thus, we are interested in the infinite case. \textbf{In what follows} we assume that $G$ is an \textbf{infinite} group.

\begin{lemma}
    Let $g \in G$ such that $|g| = \infty$. Then, $\operatorname{int} \operatorname{Fix}^{\sigma}_g = \varnothing$.
    \label{l: infinite order elements has empty interior}
\end{lemma}

\begin{proof}
    By Remark~\ref{r: invariant g point charaterization} we know that if $|g| = \infty$, then $\xi$ is an infinite set for all $\xi \in \operatorname{Fix}^{\sigma}_g$. Thus, $D(u) \nsubseteq \operatorname{Fix}^{\sigma}_g$ for all idempotent $u \in B^{\sigma}$, since $D(u)$ always has finite elements due to Lemma~\ref{l: minimal and maximal elements of Du} $(i)$. Henceforth, $\operatorname{int} \operatorname{Fix}^{\sigma}_g = \varnothing$. 
\end{proof}

An immediate consequence of Lemma~\ref{l: infinite order elements has empty interior} is the following proposition.

\begin{proposition}\label{prop:torsion-free}
    If $G$ is a torsion-free group. Then, the spectral partial action of $G$ is topologically free.
\end{proposition}

Henceforth, we are interested in the case when $G$ has torsion elements.
\begin{lemma} \label{l: xi U h not fixed}
    Let $\xi \in \Omega_{\sigma}$ and $h \in G \setminus \xi$. Then,
    
    \[
        \xi \in \operatorname{Fix}^{\sigma}_g \Rightarrow \xi \cup \{ h \} \notin \operatorname{Fix}^{\sigma}_g,
    \]
    for all $g \in G \setminus \{ 1 \}$.
    
\end{lemma}
\begin{proof}
    First note that if $\xi \cup \{ h \} \notin \Omega_{\sigma}$ it is immediate that $\xi \cup \{ h \} \notin \operatorname{Fix}^{\sigma}_{g}$.
    Suppose that $\xi \cup \{ h \} \in \operatorname{Fix}^{\sigma}_g$, then $gh \in \xi \cup \{ h \}$.
    Hence, $gh \in \xi$ and $h \notin \xi$.
    If $\xi \in \operatorname{Fix}^{\sigma}_{g}$, then $\xi \in \operatorname{Fix}^{\sigma}_{g^{-1}}$, and therefore $h \in \hat{\theta}_{g^{-1}}(\xi) = \xi$ which is a contradiction.
\end{proof}
A direct consequence of Lemma~\ref{l: xi U h not fixed} is the following:
\begin{corollary} \label{c: nonfixed opens}
    Let $U$ be an open subset of $\Omega_\sigma$.
    If there exists $\xi \in U$ and $h \in G \setminus \xi$ such that $\xi \cup \{ h \} \in U$.
    Then,  $U \nsubseteq \operatorname{Fix}^{\sigma}_g$.
\end{corollary}

Let $\alpha \in \Omega$, $\gamma \subseteq G \setminus \{ 1 \}$ such that $\alpha$ and $\gamma$ are finite sets, and $\alpha \cap \gamma = \varnothing$. We define 
\[
    (\alpha, \gamma) := \{ \xi \in \Omega : \alpha \subseteq \xi \text{ and } \xi \cap \gamma = \varnothing \},
\]
and 
\[
    (\alpha, \gamma)^{\sigma} := (\alpha, \gamma) \cap \Omega_\sigma.
\]
Notice that $\{ (\alpha, \gamma) \}$ is a basis for the topology of $\Omega$ consisting of open and compact sets. Thus, $\{ (\alpha, \gamma)^{\sigma} \}$ is a basis for the induced topology of $\Omega_{\sigma}$. Furthermore, it is easy to see that
\[
    (\alpha, \gamma)^{\sigma} \neq \varnothing \Longleftrightarrow \alpha \in \Omega_{\sigma}.
\]
\begin{proposition} \label{p: intFix not empety if contains isolated points}
    Let $g \in G$, such that $g^{n}=1$. If there exists $(\alpha, \gamma)^{\sigma} \neq \varnothing$, such that $(\alpha, \gamma)^{\sigma} \subseteq \operatorname{Fix}^{\sigma}_g$. Then,
    \begin{enumerate}[(i)]
        \item There exists $\{ h_1, \ldots, h_m \} \subseteq G$, such that $\alpha = \{ g^{s} h_i : 0 \leq s \leq n, 1 \leq i \leq m\}$.
        \item $(\alpha, \gamma)^{\sigma} = \{ \alpha \}$.
    \end{enumerate}
    Consequently, $\operatorname{int} \operatorname{Fix}^{\sigma}_g \neq \varnothing$ if, and only if, there exists $\{ h_1, \ldots, h_m \} \subseteq G$ such that $\xi = \bigcup_{1 \leq i \leq m} \{ g^{s}h_i \}^{n}_{s = 1}$ is an isolated point of $\Omega_\sigma$.
\end{proposition}
\begin{proof}
    For $(i)$ recall that, by definition, $\alpha$ is finite.
    Since $\alpha \in \operatorname{Fix}^{\sigma}_g$, then $g^{s}h \in \alpha$ for all $s \in \mathbb{Z}$ and $h \in \alpha$.
    Whence, we conclude that the desired set $\{ h_i \}^{m}_{i = 1}$ there exists.
    Now for $(ii)$, suppose that there exists $\xi \in (\alpha, \gamma)^{\sigma}$, $\xi \neq \alpha$.
    Let $h \in \xi \setminus \alpha$, then $\alpha \cup \{ h \} \in (\alpha, \gamma)^{\sigma}$ since $h \notin \gamma$.
    Thus, by Corollary~\ref{c: nonfixed opens} we have that $(\alpha, \gamma)^{\sigma} \nsubseteq \operatorname{Fix}^{\sigma}_g$, which is a contradiction.
\end{proof}

\begin{corollary}
    If $\Omega_\sigma$ has no isolated points, then the spectral partial action $\hat{\theta}$ is topologically free.
        \label{c: no isolated points implies topologically free}
\end{corollary}

Corollary~\ref{c: no isolated points implies topologically free} helps us to understand the importance of studying the isolated points of $\Omega_\sigma$. In particular, one can conclude that if $\Omega_\sigma$ is open in $\Omega$, then $\Omega_\sigma$ has no isolated points since $\Omega$ has no isolated points. 

\begin{proposition}
    Let $\sigma$ be a non-global $2$-cocycle. Then, $\operatorname{int} \Omega_\sigma = \varnothing$ considering $\Omega_{\sigma}$ as a subspace of $\Omega$.
\end{proposition}

\begin{proof}
    Since, $\sigma$ is a non-global $2$-cocycle, we have that $\mathcal{P}_{\sigma} \neq \varnothing$. Let $\{ g_1, g_2, g_3, g_4 \} \in \mathcal{P}_\sigma$, and $\alpha, \beta \subseteq G$ finite subsets such that $1 \in \alpha$ and $\alpha \cap \beta = \varnothing$, thus $(\alpha, \beta)$ is a cylinder of $\Omega$. Since, $\beta$ is finite and $G$ infinite there exists $h \in G$ such that $\beta \cap \{ h g_1, h g_2, h g_3, h g_4 \} = \varnothing$. Therefore, $\alpha \cup \{h g_1, h g_2, h g_3, h g_4 \} \in (\alpha , \beta)$, but $\alpha \cup \{ h g_1, h g_2, h g_3, h g_4 \} \notin \Omega_{\sigma}$. Thus, $(\alpha, \beta) \nsubseteq \Omega_{\sigma}$ for any cylinder $(\alpha, \beta)$ of $\Omega$.
\end{proof}

\begin{corollary}
    The following statements are equivalent:
    \begin{enumerate}[(i)]
        \item $\operatorname{int} \Omega_{\sigma} \neq \varnothing$,
        \item $\sigma$ is a global $2$-cocycle,
        \item $\Omega_\sigma = \Omega$.
    \end{enumerate}
\end{corollary}

\begin{lemma}
    If $\xi \in \Omega_{\sigma}$ is infinite then $\xi$ is not an isolated point.
        \label{l: isolated points are finite}
\end{lemma}

\begin{proof}
    Let $\alpha, \gamma \subseteq G$ (finite sets) such that $\xi \in (\alpha, \gamma)^{\sigma}$. Then, $\xi \setminus \alpha$ is infinite. Hence, for all $h \in \xi \setminus \alpha$ we have that $\xi \setminus \{ h \} \in (\alpha, \gamma)^{\sigma}$. Therefore, any open subset that contains $\xi$ has infinite many points.
\end{proof}

\begin{definition}
    For $\xi \in \Omega_{\sigma}$, we define de set of admissible elements for $\xi$ by
    \[
        A_{\sigma}(\xi) := \{ g \in G \setminus \xi : \xi \cup \{ g \} \in \Omega_{\sigma}\}.
    \]
\end{definition}

\begin{proposition} \label{p: charaterization isolated points}
    Let $\xi \in \Omega_\sigma$. Then, $\xi$ is an isolated point if, and only if, $\xi$ and $A_\sigma(\xi)$ are finite sets. Furthermore, if $\xi$ is an isolated point, then $\{ \xi \} = (\xi, A_\sigma(\xi))^{\sigma}$.
\end{proposition}
\begin{proof}
    By Lemma~\ref{l: isolated points are finite} we may assume that $\xi$ is finite. Suppose that $\xi$ is an isolated point. Then, there exists finite sets $\alpha, \gamma \in \Omega$ such that $\{ \xi \} = (\alpha, \gamma)^{\sigma}$. Hence, $(\xi, \gamma)^{\sigma} = (\alpha, \gamma)^{\sigma}$. Moreover, $A_{\sigma}(\xi) \subseteq \gamma$. Indeed, if there exists $h \in A_{\sigma}(\xi) \setminus \gamma$, then $\xi \cup \{ h \} \in (\xi, \gamma)^{\sigma}$, which is a contradiction. Thus, $A_{\sigma}(\xi)$ is finite. In the opposite direction, it is clear that if $A_{\sigma}(\xi)$ is finite then $\{ \xi \} = (\xi, A_\sigma(\xi))^{\sigma}$.
\end{proof}

\begin{corollary}
    $\Omega_{\sigma}$ has isolated points if, and only if, there exists maximal elements of $\Omega_{\sigma}$ which are finite.
\end{corollary}

\begin{corollary} \label{c: spectre is discrete if the admisibles of 1 is finite}
    $\Omega_{\sigma}$ is discrete if, and only if, $A_\sigma(\{ 1 \})$ is finite. Consequently, $\Omega_{\sigma}$ is finite if, and only if, $\Omega_{\sigma}$ is discrete.
\end{corollary}
\begin{proof}
    If $\Omega_{\sigma}$ is discrete, then $\{ 1 \}$ is isolated and by Proposition~\ref{p: charaterization isolated points} the set $A_{\sigma}(\{ 1 \})$ is finite. Conversely, note that for every $\xi \in \Omega_{\sigma}$ we have that $\xi, A_{\sigma}(\xi) \subseteq \{ 1 \} \cup A_{\sigma}(\{ 1 \})$, and again by Proposition~\ref{p: charaterization isolated points} we conclude that $\xi$ is an isolated point.
\end{proof}

\begin{proposition} \label{p: fixed point are non empty if is locally global}
    $\operatorname{Fix}^{\sigma}_{g} \neq \varnothing$ if, and only if, $\sigma|_{\langle g \rangle}$ is a global $2$-cocycle.
\end{proposition}
\begin{proof}
    Suppose that $\operatorname{Fix}^{\sigma}_{g} \neq \varnothing$, then $\langle g \rangle \in \Omega_{\sigma}$, therefore $\sigma(g^{t}, g^{s}) \neq 0$ for all $s,t \in \mathbb{Z}$ and $\sigma(g^{s}, g^{t+m})\sigma(g^{t}, g^{m}) = \sigma(g^{s}, g^{t})\sigma(g^{s + t}, g^{m})$ for all $s,t,m \in \mathbb{Z}$. Hence, $\sigma|_{\langle g \rangle}$ is a global $2$-cocycle. Conversely, suppose that $\sigma|_{\langle g \rangle}$ is a global $2$-cocycle. Let $\{ h, hx, hxy, hxyz \} \subseteq \langle g \rangle$, then $h = g^{m_1}$, $x = g^{m_2}$, $y = g^{m_3}$, $z = g^{m_4}$. Thus, $\{ h, hx, hxy, hxyz \} \notin \mathcal{P}_\sigma$ since $\sigma|_{\langle g \rangle}$ is global. Therefore, $\langle g \rangle \in \operatorname{Fix}^{\sigma}_g$. 
\end{proof}

\begin{example}
    Let $G_1$ be a finite group and $G_2$ an infinite group.
    Let $\sigma_1: G_{1} \times G_{1} \to \kappa$ be a global $2$-cocycle of $G_1$ and $\sigma_2 \in pm(G_2)$.
    We define $G := G_1 \times G_2$ and $\sigma: G \times G \to \kappa$ by
    \[
        \sigma((g,g'), (h, h')) := \left\{\begin{matrix}
            \sigma_1(g,h)  & \text{ if } g'=h'=1, \\ 
            \sigma_2(g',h')  & \text{ if } g=h=1, \\ 
            0 & \text{ otherwise}.
        \end{matrix}\right.
    \]
    Notice that, for all $g \in G_1$, by Proposition~\ref{p: fixed point are non empty if is locally global} we have that $\langle (g,1) \rangle \in \operatorname{Fix}^{\sigma}_{(g,1)} \subseteq \Omega_{\sigma}$.
    Moreover, it follows that $A_{\sigma}(\langle (g,1) \rangle)$ is finite since $\{ (h,1), (u,v) \} \in \mathcal{P}_{\sigma}$, for all $h,u \in G_1$ and $v \in G_2 \setminus \{ 1 \}$. 
    Hence, $\langle (g,1) \rangle$ is an isolated point by Proposition~\ref{p: charaterization isolated points}.
    Therefore, $\operatorname{int}(\operatorname{Fix}^{\sigma}_{(g,1)}) \neq \varnothing$ for all $g \in G_1 \setminus \{ 1 \}$.
\end{example}

From Proposition~\ref{p: intFix not empety if contains isolated points}, Proposition~\ref{p: charaterization isolated points} we obtain the following fact.

\begin{theorem} \label{t: not topologically free charaterization}
    Let $G$ be an infinite group and $\sigma \in pm(G)$. Then, the following are equivalent    
    \begin{enumerate}[(i)]
        \item The spectral partial action $\hat{\theta}$ is not topologically free,
        \item There exists an element $g \in G$ such that $\operatorname{Fix}^{\sigma}_g$ contains isolated points of $\Omega_{\sigma}$.
        \item There exists $g \in G$ and $\xi \in \operatorname{Fix}^{\sigma}_g$ such that $|\xi| < \infty$ and $|A_{\sigma}(\xi)| < \infty$.
    \end{enumerate}
\end{theorem}

\begin{remark}
    Another important property of topological partial group actions is their minimality. 
    We recall from \cite[Definition 29.7]{E6} that a topological partial group action \(\beta\) is called \textbf{minimal} if there is no \(\beta\)-invariant closed subset of \(X\) other than \(\varnothing\) and \(X\).
    It is easy to see that the spectral partial action \(\hat{\theta}\) is not minimal since \(\{ 1 \}\) is a closed point of \(\Omega_{\sigma}\) which is also
      \(\hat{\theta}\)-invariant.
\end{remark}

\section{Idempotent factor sets}\label{sec:Idempotent}
In this section we will work with idempotent ($\sigma(g,h) \in \{ 0, 1 \}$) partial $2$-cocycles of an infinite group $G$ with non-trivial torsion group.

\begin{lemma}\label{sec:OnlyTypeI}
    If $\sigma$ is idempotent then $\xi \notin \Omega_{\sigma}$ if, and only if, $\xi$ contains a $\sigma$-prohibition of type $I$.
\end{lemma}
\begin{proof}
    By definition if $\xi$ contains a $\sigma$-prohibition of type $I$ we have that $\xi \notin \Omega_{\sigma}$.
    Conversely, suppose that there exists $V \in \mathcal{P}_\sigma$ such that $V \subseteq \xi$, if $V$ is of type $I$ we are done, thus suppose that $V$ is of type $II$.
    Then, there exists $h,x,y,z \in G$ such that $V = \{ h, hx, hxy, hxyz \}$ and $\sigma(x,y)\sigma(xy,z) \neq \sigma(x,yz)\sigma(y,z)$.
    Since, $\sigma$ is idempotent then, at least one of the elements of $\{ \sigma(x,y), \sigma(xy,z), \sigma(x,yz), \sigma(y,z) \}$ is zero.
    This means that at least one of the sets $ \{ h, hx, hxy \}, \{ h, hxy, hxyz \}, \{ h, hx, hxyz \}, \{ h, hy, hyz \}$ is a $\sigma$-prohibition of type $I$.
    Since $\xi$ contains all the previous sets, it also includes a type $I$ prohibition.
\end{proof}

\subsection{Diagonal idempotent factor sets}\label{sec:DiagonalIdempotent}

Let $S \subseteq G$ be such that 
\begin{equation}
    1 \notin S \text{ and } g^{-1} \in S \text{ for all }g \in S.
    \label{eq: Diagonal condition}
\end{equation}
    
We define
\[
    \sigma_S(x,y) :=\left\{\begin{matrix}
        0 & \text{ if } \{ x, y, xy \} \cap S \neq \varnothing\\ 
        1 & \text{ otherwise.}
    \end{matrix}\right.
\]

Let $(x,y)$ be such that $\sigma_S(x,y) = 1$, then $\{ x, y, xy, x^{-1}, y^{-1}, y^{-1}x^{-1} \} \cap S = \varnothing$. Thus, $\sigma_{S}(x,1) = \sigma_{S}(xy,y^{-1}) = \sigma_{S}(y^{-1},x^{-1}) = 1$. Therefore, by \cite[Theorem 4]{DN} we conclude that $\sigma_{S} \in pm(G)$. We call $\sigma_{S}$ the \textbf{diagonal factor set} generated by $S$. Thus, we obtain the following fact.

\begin{proposition} \label{p: diagonal factor set}
    Let $S \subseteq G$, such that $g^{-1} \in S$ for all $g \in S$ and $1 \notin S$. Then, $\sigma_{S} \in pm(G)$.
\end{proposition}

\begin{lemma} \label{l: omega sigma S complement charaterization}
    Let $\xi \in \Omega$. Then, $\xi \notin \Omega_{\sigma_{S}}$ if, and only if, there exists $h \in G$ and $g \in S$ such that $\{ h, hg \} \subseteq \xi$.
\end{lemma}
\begin{proof}
    Suppose that $g \in S$ and $h \in G$ are such that $\{ h, hg \} \subseteq \xi$.
    Therefore, $\sigma_{S}(1,g)=0$, so $\{ h, hg \}= \{ h, h\cdot 1, h \cdot 1 \cdot g \} \in \mathcal{P}_{\sigma}^{\mathrm{I}}$. 
    Thus, $\xi \notin \Omega_{\sigma_{S}}$. Conversely, suppose that there exists $h, x, y \in G$ such that $\sigma_{S}(x,y) = 0$ and $\{ h, hx, hxy \} \subseteq \xi$. Hence, $\{ x, y, xy \} \cap S \neq \varnothing$, consequently at least one of the sets $\{ h, hx \}, \,\{ hx, hxy \}, \, \{ h, hxy \}$ is of the desired form and is contained in $\xi$.
\end{proof}

\begin{proposition} \label{p: omega sigma S complement charaterization}
    Let $\xi \in \Omega_{\sigma_{S}}$, then $\big(A_{\sigma_{S}}(\xi)\big)^{c}= \{ xg : x \in \xi \text{ and } g \in S \}$.
\end{proposition}
\begin{proof}
    By Lemma~\ref{l: omega sigma S complement charaterization} we have that
    \begin{align*}
        y \in \big(A_{\sigma_{S}}(\xi)\big)^{c} & \Leftrightarrow \exists x \in \xi \text{ such that } \{ x, y \} \in \mathcal{P}_{\sigma_{S}}^{\operatorname{I}} \\
        &\Leftrightarrow \exists x \in \xi, \, g \in S, \, h \in G \text{ such that } \{ x, y \} = \{ h, hg \} \\
        & \Leftrightarrow \exists g \in S \text{ and } x \in \xi \text{ such that } y = xg.
    \end{align*}
\end{proof}

\begin{remark} \label{r: complement admissibles of xi cardinal S}
    Let $\xi \in \Omega_{\sigma_{S}}$ be such that $\xi$ is finite. By Proposition~\ref{p: omega sigma S complement charaterization} we have that if $S$ is finite then $\big(A_{\sigma_{S}}(\xi) \big)^{c}$ is finite. Analogously, if $S$ is infinite then $|S| = |\big(A_{\sigma_{S}}(\xi) \big)^{c}|$.
\end{remark}

\begin{proposition}\label{prop:DiagonalTopolFree}
    Let $S$ be as above, if $|S| < |G|$ then $\Omega_{\sigma_{S}}$ has no isolated points. Consequently, the partial action $\hat{\theta}$ is topologically free.
\end{proposition}
\begin{proof}
     By Remark~\ref{r: complement admissibles of xi cardinal S} and $|S| < |G|$  we conclude that $|A_{\sigma_{S}}(\xi)| = |G|$ for any finite $\xi \in \Omega_{\sigma_{S}}$. Thus, by Proposition~\ref{p: charaterization isolated points} we conclude that $\Omega_{\sigma_{S}}$ has no isolated points and consequently by Theorem~\ref{t: not topologically free charaterization} the partial action is topologically free.
\end{proof}

\subsection{Lateral idempotent factor sets}\label{sec:LateralIdempotent}

For any pair of elements $a,b \in G$ we define 
\[
    C_{(a,b)} := \{ (a,b), (b^{-1}a^{-1}, a), (b, b^{-1}a^{-1}), (b^{-1}, a^{-1}), (a^{-1}, ab), (ab, b^{-1})\}. 
\]

Let $W \subseteq G \times G$ such that 
\begin{equation}
    C_{(z,z^{-1})} \cap W = \varnothing \text{ for all } z \in G.
    \label{eq: Lateral condition}
\end{equation}
We define,
\begin{equation} \label{eq: lateral partial factor set}
    \sigma_{W}(a,b) := \left\{\begin{matrix}
        0 & \text{ if } C_{(a,b)} \cap W \neq \varnothing \\ 
        1 &  \text{ otherwise.}
    \end{matrix}\right.
\end{equation}

\begin{proposition} \label{p: lateral factor set}
    $\sigma_{W} \in pm(G)$.
\end{proposition}
\begin{proof}
    Let $(x,y)$ such that $\sigma_{W}(x,y) = 1$, consequently $\sigma_{W}(y^{-1}, x^{-1}) =\sigma_{W}(xy, y^{-1}) = 1$. On the other hand, $\sigma_{W}(x,1) = 1$ since $ C_{(x, x^{-1})} \cap W = \varnothing$. Hence, by \cite[Theorem 4]{DN} we have $\sigma_{W} \in pm(G)$. We call $\sigma_{W}$ as the \textbf{lateral factor set} generated by $W$.
\end{proof}

\begin{lemma} \label{l: lateral factor set prohibitions charaterization}
    $\xi \notin \Omega_{\sigma_{W}}$ if, and only if, there exists $h \in G$ and $(x,y) \in W$ such that $\{ h, hx, hxy \} \subseteq \xi$.
\end{lemma}

\begin{proof}
    Suppose that $\xi \notin \Omega_{\sigma_{W}}$, then there exists $h, x, y \in G$ such that $\sigma_{W}(x,y)=0$ and $\{ h, hx, hxy \} \subseteq \xi$. Thus, $C_{(x,y)} \cap W \neq \varnothing$, therefore:
    \begin{enumerate}[(i)]
        \item if $(x,y) \in W$ we are done, 
        \item if $(y^{-1}x^{-1}, x) \in W$, then $\{ hxy, hxy (y^{-1} x^{-1}), hxy (y^{-1}x^{-1}) x \} = \{ h, hx, hxy \} \subseteq \xi $,        
        \item if $(y, y^{-1}x^{-1}) \in W$, then $\{ hx, hxy, hxy(y^{-1}x^{-1}) \} = \{ h, hx, hxy \} \subseteq \xi$,
        \item if $(y^{-1},x^{-1}) \in W$, then $\{ hyx, (hxy)y^{-1}, (hxy)y^{-1}x^{-1} \} = \{ hxy, hx, h \} \subseteq \xi$,
        \item if $(x^{-1}, xy) \in W$, then by items $(iv)$ and $(ii)$ we obtain the desired conclusion,
        \item if $(xy, y^{-1}) \in W$, then our conclusion follows from items $(iv)$ and $(iii)$.
    \end{enumerate}
\end{proof}

By Lemma~\ref{l: lateral factor set prohibitions charaterization} we conclude that 
\[
    \big(A_{\sigma_{W}}(\xi)\big)^{c} = \big\{ z \in G \setminus \xi: \text{ there exists } h \in G, \, (x,y) \in W \text{ such that } \{ h, hx, hxy \} \subseteq \xi \cup \{ z \} \big\}.
\]
Thus,
\begin{align*}
    \big(A_{\sigma_{W}}(\xi)\big)^{c} = \{ &hxy : h, hx \in \xi \text{ and } (x,y) \in W \} \\ 
    & \cup \{ hx : h, hxy \in \xi \text{ and } (x,y) \in W \} \\ 
    & \cup \{ h : hx, hxy \in \xi \text{ and } (x,y) \in W\}.
\end{align*}
Whence one conclude that if $\xi \in \Omega_{\sigma_{W}}$ and $W$ are finite, then $\big( A_{\sigma_{W}}(\xi)\big)^{c}$ is finite. Moreover, if $\xi$ is finite and $W$ is infinite, then $|\big( A_{\sigma_{W}}(\xi)\big)^{c}| \leq |W|$. Therefore, we obtain the following fact:

\begin{proposition}\label{prop:LateralTopolFree}
    If $|W| < |G|$, then $\Omega_{\sigma_{W}}$ has no isolated points, and consequently the partial action $\hat{\theta}$ is topologically free.
\end{proposition}

\subsection{General idempotent factor sets}

Let $T \subseteq G \times G$ such that if 
\begin{equation}
    (x,y) \in T \text{ then } (y^{-1}, x^{-1}) \in T \text{ and } (1,1) \notin T.
    \label{eq: General idempotent condition}
\end{equation}
Define 
\[
    T_1 = \{ (x,y) \in T : x \neq 1, \, y \neq 1, \text{ and } x \neq y^{-1}  \},
\]
and
\[
    T_0 = \big\{ x \in G \setminus \{ 1 \} : C_{(x,x^{-1})} \cap T \neq \varnothing \big\}.
\]
Notice that $T_0$ satisfies condition \eqref{eq: Diagonal condition} and $T_1$ satisfies condition \eqref{eq: Lateral condition}.
Define $\sigma^{0} := \sigma_{T_0}$, $\sigma^{1} := \sigma_{T_1}$ as in Proposition~\ref{p: diagonal factor set} and Proposition~\ref{p: lateral factor set} respectively, and $\sigma_{T} := \sigma^{0} \sigma^{1}$. We call $\sigma_{T}$ as the \textbf{idempotent factor set} generated by $T$. 

\begin{lemma} \label{l: any idempotent factor set is generated by Null}
    Let $\sigma$ be an idempotent factor set, and
    \[
       \operatorname{Null}(\sigma):= \{ (x,y) \in G \times G : \sigma(x,y) = 0 \}.
    \]
    Then, $\sigma$ is the idempotent factor set generated by $\operatorname{Null}(\sigma)$.

\end{lemma}
\begin{proof}
    Set $T := \operatorname{Null}(\sigma)$. 
    Note that by \cite[Corollary 7]{DN} if $(x,y) \in T$, then $C_{(x,y)} \subseteq T$.
    In order to show that $\sigma = \sigma_{T}$ we only have to verify that $\sigma_{T}(x,y) = 0$ if, and only if, $(x,y) \in \operatorname{Null}(\sigma)$. It is clear that if $(x,y) \in T$ then $\sigma_{T}(x,y) = 0$. On the other hand, suppose that $\sigma_{T}(x,y) = 0$.
    Then there are two possible cases:
    \begin{enumerate}[(i)]
        \item $(x,y) \in T_{1} \subseteq T = \operatorname{Null}(\sigma)$,
        \item $\{ x, y, xy \} \cap T_{0} \neq \varnothing$, this means that 
            \[
                \sigma(x,1) = 0 \text{ or } \sigma(y,1) = 0 \text{ or }\sigma(xy,1) = 0,             
            \]
            which implies by \cite[Proposition 4 and Proposition 5]{DN} that $\sigma(x,y)=0,$ and then $(x,y) \in \operatorname{Null}(\sigma)$.
    \end{enumerate}
\end{proof}

We proceed with the next easy fact:
\begin{lemma}
    Let $\sigma$ and $\rho$ be idempotent factor sets, then $\Omega_{\sigma \rho} = \Omega_{\sigma} \cap \Omega_{\rho}$.
\end{lemma}
\begin{proof}
    Observe that
    \begin{align*}
        \xi \in \Omega_{\sigma \rho} 
        & \Leftrightarrow \text{ for all } h, x, y \in G \text{ such that } \sigma(x,y)\rho(x,y) = 0 \text{ we have that }  \{ h, hx, hxy \} \nsubseteq \xi \\
        & \Leftrightarrow \text{ for all } h, x, y \in G \text{ such that } \sigma(x,y)=0 \text{ or } \rho(x,y) = 0 \text{ we have that }  \{ h, hx, hxy \} \nsubseteq \xi \\
        & \Leftrightarrow \text{ for all } V \in \mathcal{P}^{I}_{\sigma} \text{ and } V' \in \mathcal{P}^{I}_{\rho} \text{ we have that } V, V' \nsubseteq \xi \\
        & \Leftrightarrow \xi \in \Omega_{\sigma} \cap \Omega_{\rho}.
    \end{align*}
    
\end{proof}

\begin{proposition} \label{p: small generated set of a factor set implies no isolated points}
    Let $T \subseteq G \times G$ be a set satisfying condition \eqref{eq: General idempotent condition} such that $|T| < |G|$.
    Then, $\Omega_{\sigma_{T}}$ has no isolated points.
\end{proposition}

\begin{proof}
    Observe that
    \begin{align*}
        \big( A_{\sigma_{T}}(\xi) \big)^{c} 
        &= \big\{ h : \xi \cup \{ h \} \notin \Omega_{\sigma^{0}} \cap \Omega_{\sigma^{1}}  \big\} \\
        &= \big\{ h : \xi \cup \{ h \} \notin \Omega_{\sigma^{0}} \text{ or } \xi \cup \{ h \} \notin \Omega_{\sigma^{1}} \big\} \\
        &= \big\{ h : h \in ( A_{\sigma^{0}}(\xi) )^{c} \text{ or } h \in ( A_{\sigma^{1}}(\xi) )^{c} \big\} \\
        &= \big( A_{\sigma^{0}}(\xi) \big)^{c} \cup \big( A_{\sigma^{1}}(\xi) \big)^{c}.
    \end{align*}
    Since, $|T_0| < |T|$ and $|T_1| < |T|$ we have that if $\xi$ is finite then  $|(A_{\sigma^{0}}(\xi))^{c}|, |(A_{\sigma^{1}}(\xi))^{c}| < |G|$, consequently $|(A_{\sigma_{T}}(\xi))^{c}| < |G|$ and $|A_{\sigma_{T}}(\xi)| = |G|$.
\end{proof}

A direct consequence of Theorem~\ref{t: not topologically free charaterization} and Proposition~\ref{p: small generated set of a factor set implies no isolated points} is the following corollary:

\begin{corollary}\label{cor:IdempotentTopolFree}
     Suppose that $T \subseteq G \times G$ satisfies condition \eqref{eq: General idempotent condition},
    and let $\sigma$ be a factor set generated by $T.$ If  $|T| < |G|,$
    then $\hat{\theta}$ is topologically free.
\end{corollary}

The next fact is direct consequence of Lemma~\ref{l: any idempotent factor set is generated by Null} and Corollary~\ref{cor:IdempotentTopolFree}. 

\begin{corollary}\label{cor:IdempotentNullTopolFree}
    Let $\sigma$ be an idempotent factor set of $G$.
    If $|\operatorname{Null}(\sigma)| < |G|,$ then $\hat{\theta}$ is topologically free.
\end{corollary}

\begin{example}
    Let $G$ be a group and $N$ a subgroup of $G$. Define
    \[
        \sigma_{N}(a,b) := \left\{\begin{matrix}
            1 & \text{ if }  a,b \in N \\ 
            0 &  \text{ otherwise.}
        \end{matrix}\right.
    \]
    Then, by \cite[Theorem 4]{DN} we know that $\sigma_{N} \in pm(G)$. If $N$ is infinite, then $\Omega_{\sigma_{N}}$ has no isolated points. Indeed, notice that $\xi \subseteq N$ for all $\xi \in \Omega_{\sigma_{N}}$, and $A_{\sigma_{N}}(\xi) = N \setminus \xi$ is infinite if $\xi$ is finite. On the other hand, if $N$ is finite then by Corollary~\ref{c: spectre is discrete if the admisibles of 1 is finite} we have that $\Omega_{\sigma_{N}}$ is finite and discrete.
\end{example}

\begin{theorem} \label{p: unique decomposition of idempotent factor sets}
    Let $\sigma$ be an idempotent factor set of $G$. Then, there exist unique idempotent factor sets $\delta$ and $\lambda$ such that:
    \begin{enumerate}[(i)]
        \item $\delta$ is a diagonal factor set;
        \item $\lambda$ is a lateral factor set;
        \item $\sigma = \delta \lambda$;
        \item if $\delta(x, y) = 0$, then $\lambda(x, y) = 1$,
        \item $\operatorname{Null}(\delta) \cap \operatorname{Null}(\lambda) = \varnothing$ and $\operatorname{Null}(\sigma) = \operatorname{Null}(\delta) \sqcup \operatorname{Null}(\lambda)$.
    \end{enumerate}
\end{theorem}
\begin{proof}
    By Lemma~\ref{l: any idempotent factor set is generated by Null} we know that $\sigma$ is generated by $\operatorname{Null}(\sigma)$. Thus, $\sigma = \sigma_{S} \sigma_{W}$ where $\sigma_{S}$ is the diagonal factor set generated by
    \[
        S := \{ x \in G \setminus \{ 1 \} : C_{(x,x^{-1})} \cap \operatorname{Null}(\sigma) \neq \varnothing \},
    \]
    and $\sigma_{W}$ is the lateral factor set generated by
    \[
        W := \{ (x,y) \in \operatorname{Null}(\sigma) : x \neq 1, \, y \neq 1 \text{ and } x \neq y^{-1} \}.
    \]
    We define a substitute for $W$ as follows:
    \[
        \overline{W}:= \{ (x,y) \in W : \sigma_S(x,y) = 1 \}.
    \]
    Note that by construction, if $\sigma_W(x,y) = 1$, then $\sigma_{\overline{W}}(x,y)=1$.
    Moreover, $\sigma_{\overline{W}}(x,y) = 0$ if, and only if, $\sigma(x,y) = 0$ and $\sigma_{S}(x,y) = 1$.
    Thus, $\sigma = \delta \lambda$ where $ \delta = \sigma_{S}$ and $\lambda = \sigma_{\overline{W}}$.
    This gives $(iv)$ which implies $(v)$ and the uniqueness of $\delta$ and $\lambda$.
\end{proof}

\section[Infinite dihedral group]{Infinite dihedral group \(D_{\infty}\)}\label{sec:InfDihedral}

By \cite[Corollary 5.8]{DoNoPi}  any factor set is equivalent to the pro\-duct of an idempotent factor set a factor set with total domain.
Furthermore, by Proposition~\ref{p: unique decomposition of idempotent factor sets} we know that any idempotent factor set is the product of a diagonal factor set and a lateral factor set. 

For the particular case of the infinite dihedral group $D_{\infty}$ by \cite[Theorem 5.4]{HGGLimaHPinedo} we know that the factor sets with total domain (up to equivalence) are isomorphic to the semigroup \((\kappa^*)^{\mathbb{Z}^{+} \times \mathbb{Z}^{+} \sqcup \mathbb{Z}^{+} \times \mathbb{Z}}\).
Therefore, a direct application of the previous results about lateral and diagonal factor sets is to fully describe, up to equivalence, the factor sets of the infinite dihedral group $D_{\infty}.$

Denote by $D_{\infty}$ the infinite dihedral group with presentation
\[
    D_{\infty}= \langle a, b | b^{2}=1, \, ba = a^{-1}b \rangle.
\]

\begin{remark}
    Let $\sigma$ be a factor set of $G$, and let $\xi \in \operatorname{Fix}^{\sigma}_{ba^{l}}$. Observe that for $k \in \mathbb{Z}$ we have that
    \[
       \xi \cup \{ a^{k} \} \in \Omega_{\sigma_{\omega}} 
       \Leftrightarrow
       \xi \cup \{ ba^{k+l} \} = \hat{\theta}_{ba^{l}}(\xi \cup \{ a^{k} \}) \in \Omega_{\sigma_{\omega}}.
    \]
    Therefore, $a^{k} \in A_{\sigma}(\xi)$ if, and only if, $ba^{k+l} \in A_{\sigma}(\xi)$.
\end{remark}

\subsection{Diagonal factor sets of the infinite dihedral group}

Let $S$ be a symmetric subset of \( D_{\infty} \setminus \{ 1\}\). Observe that
\[
   D_{\infty}= \{ a^{i} \}_{i \in \mathbb{Z}} \sqcup \{ ba^{i} \}_{i \in \mathbb{Z}}.
\]

Therefore, we can split $S$ in two disjoints sets $S = S_{a} \sqcup S_{b}$ such that $S_{a} \subseteq \{ a^{i} \}$ and $S_b \subseteq \{ b a^{i} \}$. Write \(\textbf{2}:= \{ 0, 1\}\).
Thus, $S$ is determined by a function $\nu^S:=(\nu_0^{S}, \nu_1^{S}): \mathbb{Z}^{+} \sqcup \mathbb{Z} \to \textbf{2}$ such that  $\nu_0 : \mathbb{Z}^{+} \to  \textbf{2}^{S}$  is given by 
\[
    \nu_0^{S}(n):=\left\{\begin{matrix}
        0 & \text{ if }a^{n} \in S \\ 
        1 & \text{ otherwise,}
    \end{matrix}\right.  
\]
and $\nu_1^{S}: \mathbb{Z} \to \textbf{2}$ is such that
\[
    \nu_1^{S}(m):=\left\{\begin{matrix}
        0 & \text{ if }ba^{m} \in S \\ 
        1 & \text{ otherwise.}
    \end{matrix}\right.  
\]
Thus, we obtain a map from the diagonal factor sets to the functions of $\mathbb{Z}^{+} \sqcup \mathbb{Z}$ to $\textbf{2}$, given by $\sigma_{S} \to \nu^{S}$.
Moreover, we can prove that all the diagonal factor sets of $D_{\infty}$ are in bijection with $\{ 0,1 \}^{\mathbb{Z}^{+}\sqcup \, \mathbb{Z}}$.
Indeed, the inverse map of the above function is constructed as follows: for any $\nu := (\nu_0, \nu_1) \in \textbf{2}^{\mathbb{Z}^{+} \sqcup \, \mathbb{Z}}$ such that $\nu_{0}: \mathbb{Z}^{+} \to \{ 0,1 \}$ and $\nu_{1}: \mathbb{Z} \to \{ 0,1 \}$, we define 
    \[
        S_{\nu} := \{ a^{i}, a^{-i} : i \in \mathbb{Z}^{+} \text{ and } \nu_0(i)=0 \} \cup  \{ ba^{i} : i \in \mathbb{Z} \text{ and } \nu_1(i)=0 \}.
    \]
Thus, the map $\nu \mapsto \sigma_{S_{\nu}}$ is the inverse function of $\sigma_{S} \to \nu^{S}$. Hence, we obtain the following lemma:
\begin{lemma} \label{l: diagonal factor sets of DInf}
    Let $pm(D_\infty)_{\operatorname{diag}} \subseteq pm(D_{\infty})$ be the set of the diagonal factor sets. Then, $pm(D_\infty)_{\operatorname{diag}} \cong \textbf{2}^{\mathbb{Z}^{+} \sqcup \, \mathbb{Z}}$.
\end{lemma}

Let $\nu := (\nu_0, \nu_1) \in \textbf{2}^{\mathbb{Z}^{+} \sqcup \, \mathbb{Z}}$ such that $S_{\nu} = S$, by Lemma~\ref{l: omega sigma S complement charaterization} we know that the $\sigma_{S}$-prohibitions are determined by the set $\mathcal{P}(S) :=\big\{ \{ g, gs \} : g \in G \text{ and } s \in S \big\}$, therefore 
\begin{equation} \label{eq: symmetric S-prohibitions char}
    \{ x, y \} \in \mathcal{P}(S) \Leftrightarrow x^{-1}y \in S \Leftrightarrow y^{-1}x \in S.
\end{equation}
Thus, we obtain three types of $\mathcal{P}(S)$-prohibition on $D_{\infty}$:

\begin{lemma}
    Let $S \subseteq D_{\infty} \setminus \{ 1 \}$ be a symmetric subset, and $\nu: \mathbb{Z}^{+} \sqcup \mathbb{Z} \to \{ 0,1 \}$ such that $S_{\nu} = S$. Then, the $\mathcal{P}(S)$-prohibition are classified in three types:
    \begin{equation}
       \text{For }p,q \in \mathbb{Z},\, p \neq q ,\, \{ a^{p}, a^{q} \} \in \mathcal{P}(S) \Leftrightarrow a^{p-q} \in S_a \Leftrightarrow \nu_{0}(|p-q|) = 0,         
        \label{eq: I sym DInf}
    \end{equation}
    \begin{equation}
        \text{For }p,q \in \mathbb{Z},\, p \neq q ,\, \{ ba^{p}, ba^{q} \} \in \mathcal{P}(S) \Leftrightarrow ba^{p} ba^{q} = a^{q-p} \in S_a \Leftrightarrow  \nu_{0}(|p-q|) = 0, 
        \label{eq: II sym DInf}
    \end{equation}
    \begin{equation}
        \text{For }p,q \in \mathbb{Z},\, \{ a^{p}, ba^{q} \} \in \mathcal{P}(S) \Leftrightarrow a^{-p} ba^{q} = ba^{p+q} \in S_b \Leftrightarrow \nu_1(p+q) = 0
        \label{eq: III sym DInf}
    \end{equation}
\end{lemma}
\begin{proof}
    Direct computations using \eqref{eq: symmetric S-prohibitions char} and Lemma~\ref{l: diagonal factor sets of DInf}.
\end{proof}

Note that if $\xi \in \operatorname{Fix}^{\sigma}_{ba^{l}}$ then
\[
   \xi = \{ 1, a^{i_{1}}, \ldots, a^{i_{n}}, \ldots \} \cup \{ ba^{l}, ba^{i_{1}+l}, \ldots, ba^{i_{n}+l}, \ldots \},
\]

thus to each $\xi \in \operatorname{Fix}^{\sigma}_{ba^{l}}$ we can associate a set $
I_{\xi}^{l}:= \{ i \in \mathbb{Z} : a^{i} \in \xi \}$. Note that the set $I^{l}_{\xi}$ completely determines the fixed point $\xi$. Conversely, let $I \subseteq \mathbb{Z}$, $0 \in I$, then by equations \eqref{eq: I sym DInf}, \eqref{eq: II sym DInf} and \eqref{eq: III sym DInf} we obtain
\[
   \xi_I^{l} = \{ a^{i} \}_{i \in I} \cup \{ ba^{i+l} \}_{i \in I} \in \Omega_{\sigma_{S}} \Leftrightarrow 
    \nu_0(|p-q|) = 1 \text{ and } \nu_1(p+q+l)=1 \forall p,q \in I.
\]
Therefore, we define the set
\[
    \Delta_{\nu}^{l}:= \{ I \subseteq \mathbb{Z} : 0 \in I, \, \nu_0(|p-q|)=1, \, \forall p,q \in I, p \neq q, \text{ and } \nu_1(p+q+l)=1, \, \forall p, \, q \in I\}
\]
that is in bijection with the set $\operatorname{Fix}^{\sigma}_{ba^{l}}$.

\subsection{Lateral factor sets of the infinite dihedral group}

Let $W \subseteq D_{\infty} \times D_{\infty}$ be such that $C_{(z, z^{-1})} \cap W = \varnothing$ for all $z \in D_{\infty}$.

\begin{remark}
    Observe that in the definition of $\sigma_{W}$ (see equation \eqref{eq: lateral partial factor set}) if $W' \subseteq G \times G$ is such that $\{ C_{(x,y)} : (x,y) \in W \} = \{ C_{(x,y)} : (x,y) \in W'\}$ then $\sigma_{W} = \sigma_{W'}$. In particular, if we replace $(x,y) \in W$ by another element $(a,b) \in C_{(x,y)}$ we obtain the same generated factor set.
\end{remark}

\begin{remark} \label{r: S3 Orbits of Dinf}
    By \cite[Lemma 5.3]{HGGLimaHPinedo} we know that the set
    \[
        \{ C_{(x,y)} : (x,y) \in G^{2} \text{ such that } x \neq 1, \, y \neq 1 \text{ and } x \neq y^{-1} \}
    \]
    is in bijection with the set
    \[
        \{ (a^{n}, a^{m}) : (n,m) \in \mathbb{Z}^{+} \times \mathbb{Z}^{+}\} \sqcup \{ (a^{i}, ba^{j}) : (i, j) \in \mathbb{Z}^{+} \times \mathbb{Z}\}.
    \]
\end{remark}

\begin{proposition} \label{p: lateral factor sets of DInf characterization}
    The lateral factor sets of $D_{\infty}$ are isomorphic as a semigroup to
    \[
       \big\{ \omega: \mathbb{Z}^{+} \times \mathbb{Z}^{+} \sqcup \mathbb{Z}^{+} \times \mathbb{Z} \to \textbf{2} \big\}.
    \]
\end{proposition}

\begin{proof}
    Let $\omega:=(\omega_0, \omega_1): \mathbb{Z}^{+} \times \mathbb{Z}^{+} \sqcup \mathbb{Z}^{+} \times \mathbb{Z} \to \textbf{2}$ where $\omega_0: \mathbb{Z}^{+} \times \mathbb{Z}^{+} \to \textbf{2}$ and $\omega_1: \mathbb{Z}^{+} \times \mathbb{Z} \to \textbf{2}$. We define
    \begin{align*}
        W_{\omega} := \big\{ (a^{n}, a^{m}) : (n,m) \in \mathbb{Z}^{+} \times \mathbb{Z}^{+} &\text{ and } \omega_0(n,m)=0 \big\}  \\
        &\cup \big\{ (a^{i}, ba^{j}) : (i,j) \in \mathbb{Z}^{+} \times \mathbb{Z} \text{ and } \omega_1(i,j)=0 \big\}.
    \end{align*}
    Therefore, $W_{\omega}$ generates the lateral factor set $\sigma_{\omega}:= \sigma_{W_{\omega}}$. Thus, we define the map
    \[
        \phi: \big\{ \omega: \mathbb{Z}^{+} \times \mathbb{Z}^{+} \sqcup \mathbb{Z}^{+} \times \mathbb{Z} \to \textbf{2} \big\} \to pm(D_{\infty})_{\operatorname{lat}}
    \]
    such that $\phi(\omega):= \sigma_{\omega}$. By Remark~\ref{r: S3 Orbits of Dinf} we conclude that this map is a bijection. Explicitly, if $W$ is such that $C_{(z, z^{-1})} \cap W = \varnothing$ then we define
    \[
        W_0 := \{ (n,m) \in \mathbb{Z}^{+} \times \mathbb{Z}^{+} : C_{(a^{n}, a^{m})} \cap W \neq \varnothing \}
    \]
    and
    \[
        W_1 := \{ (i,j) \in \mathbb{Z}^{+} \times \mathbb{Z} : C_{(a^{i}, ba^{j})} \cap W \neq \varnothing \}.
    \]
    Thus, we define the maps $\omega_0: \mathbb{Z}^{+} \times \mathbb{Z}^{+} \to \textbf{2}$ and $\omega_1: \mathbb{Z}^{+} \times \mathbb{Z} \to \textbf{2}$ such that
    \[
        \omega_0(n,m) :=
        \left\{\begin{matrix}
            0 & \text{ if } (n,m) \in W_0 \\ 
            1 & \text{ otherwise}
        \end{matrix}\right.
    \]
    and
    \[
        \omega_1(i,j) :=
        \left\{\begin{matrix}
            0 & \text{ if } (i,j) \in W_1 \\ 
            1 & \text{ otherwise.}
        \end{matrix}\right.
    \]
    Finally, $\omega:= (\omega_0, \omega_1)$ is such that $\sigma_\omega = \sigma_W$.
\end{proof}

\begin{remark} \label{r: lateral prohibitions DInf}
    Let $\omega= (\omega_0, \omega_1)$ as above. By Proposition~\ref{l: lateral factor set prohibitions charaterization} and Proposition~\ref{p: lateral factor sets of DInf characterization} we conclude that $\xi \in \Omega_{\sigma_{\omega}}$ if, and only if, 
    \begin{enumerate}[(i)]
        \item $\{ a^{u}, a^{u + n}, a^{u + n + m} \} \nsubseteq \xi $ for all $u \in \mathbb{Z}$ and all pairs $(n,m) \in \mathbb{Z}^{+} \times \mathbb{Z}^{+}$ such that $\omega_0(n,m)=0$.
        \item $\{ ba^{u}, ba^{u + n}, ba^{u + n + m} \} \nsubseteq \xi $ for all $u \in \mathbb{Z}$ and all pairs $(n,m) \in \mathbb{Z}^{+} \times \mathbb{Z}^{+}$ such that $\omega_0(n,m)=0$.
        \item $\{ a^{u}, a^{n + u}, ba^{m - u - n} \} \nsubseteq \xi $ for all $u \in \mathbb{Z}$ and all pairs $(n,m) \in \mathbb{Z}^{+} \times \mathbb{Z}$ such that $\omega_1(n,m)=0$.
        \item $\{ ba^{u}, ba^{n + u}, a^{m - u - n} \} \nsubseteq \xi $ for all $u \in \mathbb{Z}$ and all pairs $(n,m) \in \mathbb{Z}^{+} \times \mathbb{Z}$ such that $\omega_1(n,m)=0$.
    \end{enumerate}
\end{remark}

\begin{lemma}
    Let $I, J \subseteq \mathbb{Z}$, $0 \in I$, write
    \[
       \xi:= \{ a^{i} \}_{i \in I} \cup \{ ba^{j} \}_{j \in J}.
    \]
    Then, $\xi \in \Omega_{\sigma_{\omega}}$ if, and only if,
    \begin{enumerate}[(1)]
        \item $\omega_0(y-x, z-y) = 1 \text{ for all } x,y,z \in I \text{ such that } x<y<z$,
        \item $\omega_0(y-x, z-y) = 1 \text{ for all } x,y,z \in J \text{ such that } x<y<z$,
        \item $\omega_1(y-x, z+y) = 1 \text{ for all } x,y \in I, \,  x<y$ and $z \in J$,
        \item $\omega_1(y-x, z+y) = 1 \text{ for all } x,y \in J, \,  x<y$ and $z \in I$.
    \end{enumerate}
\end{lemma}
\begin{proof}
    Item $(i)$ of Remark~\ref{r: lateral prohibitions DInf} is equivalent to
    \begin{align*}
        \{ a^{u}, &\, a^{u+n}, a^{u+n+m} \} \neq \{ a^{x}, a^{y}, a^{z} \} \, \forall u \in \mathbb{Z}, \, \forall n,m \in \mathbb{Z}^{+}, \, \omega_0(n,m)=0,\, \forall x,y,z \in I \\
        & \Leftrightarrow
        (u, u+n, u+n+m) \neq (x,y,z) \, \forall u \in \mathbb{Z}, \, \forall n,m \in \mathbb{Z}^{+}, \, \omega_0(n,m)=0,\, \forall x,y,z \in I \\
        & \Leftrightarrow
        \text{ if } (u, u+n, u+n+m) = (x,y,z) \text{ for some } u \in \mathbb{Z}, \, n,m \in \mathbb{Z}^{+} \\ & \hspace{7cm} \text{ and } x,y,z \in I, \text{ then } \omega_0(n,m)=1,   
    \end{align*}
    which is equivalent to $(1)$. Analogously, one shows that $(2)$ is equivalent to $(ii)$ of Remark~\ref{r: lateral prohibitions DInf}. Item $(iii)$ of Remark~\ref{r: lateral prohibitions DInf} is equivalent to 
   \begin{align*}
       \{ a^{u}, &a^{u+n}, ba^{m-u-n} \} \neq \{ a^{x}, a^{y}, ba^{z} \} \, \forall u \in \mathbb{Z}, \, \forall (n,m) \in \mathbb{Z}^{+} \times \mathbb{Z}, \, \omega_1(n,m)=0,
        \\ & \hspace{7cm} \forall x,y \in I, \text{ and } z \in J  \\
       & \Leftrightarrow
        (u, u+n, m-u-n) \neq (x,y,z) \, \forall u \in \mathbb{Z}, \, \forall (n,m) \in \mathbb{Z}^{+} \times \mathbb{Z}, \, \omega_1(n,m)=0, \\
            & \hspace{7cm} \forall x,y \in I, \text{ and } z \in J  \\
       & \Leftrightarrow
        \text{ if } (u, u+n, m-u-n) = (x,y,z) \text{ for some } u \in \mathbb{Z}, \, (n,m) \in \mathbb{Z}^{+} \times \mathbb{Z}, \\ & \hspace{7cm} x,y \in I \text{ and } z \in J\text{ then } \omega_0(n,m)=1,   
   \end{align*}
    which is equivalent to (3). Analogously one shows that $(iv)$ of Remark~\ref{r: lateral prohibitions DInf} is equivalent to $(4)$.
\end{proof}

\begin{corollary}
    Let $I \subseteq \mathbb{Z}$, $0 \in I$. Then, $\xi_I^{l} := \{ a^{i} \}_{i \in I} \cup \{ b a^{i+l} \}_{i \in I} \in \Omega_{\sigma_{\omega}}$ if, and only if,
    \begin{enumerate}[(I)]
        \item $\omega_0(y-x, z-y) = 1 \text{ for all } x,y,z \in I \text{ such that } x<y<z$, 
        \item $\omega_1(y-x, z+y+l) = 1 \text{ for all } x,y,z \in I \text{ such that } x<y$.
    \end{enumerate}
\end{corollary}

Similarly to the diagonal case, we define
\begin{align*}
    \Lambda^{l}_{\omega}:= \Big\{ I \subseteq \mathbb{Z} : 
    0 \in I, &\quad \omega_0(y-x, z-y) = 1 \text{ for all } x,y,z \in I \text{ such that } x<y<z, \\
    & \quad \omega_1(x-y, z+y+l) = 1 \text{ for all } x,y,z \in I \text{ such that } x<y.
                                                                            \Big\}.
\end{align*}
Observe that $\Lambda^{l}_{\omega}$ is in bijection with $\operatorname{Fix}_{ba^{l}}^{\sigma_{\omega}}$.

\subsection{Idempotent factor sets of the infinite dihedral group}

Let $\sigma$ be an idempotent factor set of $D_{\infty}$. By Proposition~\ref{p: unique decomposition of idempotent factor sets}, Lemma~\ref{l: diagonal factor sets of DInf} and Proposition~\ref{p: lateral factor sets of DInf characterization} there exists unique $\nu: \mathbb{Z}^{+} \sqcup \mathbb{Z} \to \textbf{2}$ and $\omega: \mathbb{Z}^{+} \times \mathbb{Z}^{+} \sqcup \mathbb{Z}^{+} \times \mathbb{Z} \to \textbf{2}$ such that $\sigma = \sigma_{\nu} \sigma_{\omega}$ and $\operatorname{Null}(\sigma_{\nu}) \cap \operatorname{Null}(\sigma_{\omega}) = \varnothing$. Therefore,
\[
    \mathcal{P}^{\operatorname{I}}_{\sigma} = \mathcal{P}^{\operatorname{I}}_{\sigma_{\nu}} \sqcup \mathcal{P}^{\operatorname{I}}_{\sigma_{\omega}}.
\]

\begin{proposition}\label{p: idempotent factor sets of DInf}
    There is a surjection from the semigroup of all pairs of functions \((\nu, \omega)\) where \(\nu: \mathbb{Z}^{+} \sqcup \mathbb{Z} \to \textbf{2}\) and \(\omega: \mathbb{Z}^{+} \times \mathbb{Z}^{+} \sqcup \mathbb{Z}^{+} \times \mathbb{Z} \to \textbf{2}\) and the semigroup of all idempotent factor sets of \(D_{\infty}\).
\end{proposition}

\begin{remark}
    By Proposition~\ref{p: idempotent factor sets of DInf} the task of determine when a factor set is idempotent is reduced to a pure computational problem involving functions from \(\mathbb{Z}\) to \(\textbf{2}\).
\end{remark}

Combining \cite[Theorem 5.4]{HGGLimaHPinedo} and Proposition~\ref{p: idempotent factor sets of DInf} we obtain the following:
\begin{theorem}\label{teo:InfDihedtral}
    Let \(pm'(D_{\infty})\) be the subsemigroup of \(pm(D_{\infty})\) of all factor sets such that \(\sigma(g,g^{-1}) \in \{ 0, 1\}\). Then, there is a surjective map of semigroup 
    \[
       \textbf{2}^{\mathbb{Z}^{+} \sqcup \mathbb{Z}} \times \kappa^{\mathbb{Z}^{+} \times \mathbb{Z}^{+} \sqcup \mathbb{Z}^{+} \times \mathbb{Z}} \to pm'(D_{\infty}),
    \]
    where \(\textbf{2}:= \{ 0, 1\}\). Furthermore, for every \(\sigma \in pm'(D_{\infty})\) there exists unique \(\nu: \mathbb{Z}^{+} \sqcup \mathbb{Z} \to \textbf{2}\) and \(\omega: \mathbb{Z}^{+} \times \mathbb{Z}^{+} \sqcup \mathbb{Z}^{+} \times \mathbb{Z} \to \kappa\) such that \(\sigma = \sigma_{\nu} \sigma_{\omega}\) and \(\operatorname{Null}(\sigma_{\nu}) \cap \operatorname{Null}(\sigma_{\omega}) = \varnothing\).
\end{theorem}

\section*{Acknowledgments}

The first named author was partially supported by 
Funda\c c\~ao de Amparo \`a Pesquisa do Estado de S\~ao Paulo (Fapesp), process n°:  2020/16594-0, and by  Conselho Nacional de Desenvolvimento Cient\'{\i}fico e Tecnol{\'o}gico (CNPq), process n°: 312683/2021-9. The second named author was supported by Fapesp, process n°: 2022/12963-7.

\bibliographystyle{abbrv}
\bibliography{azu.bib}

\end{document}